\documentclass{amsart}

\usepackage[english]{babel}
\usepackage[centertags]{amsmath}
\usepackage{amsfonts,amssymb,amsthm}
\usepackage{hyperref}  					 
\usepackage{mathrsfs}
\usepackage[latin1]{inputenc}
\usepackage{color}
\usepackage[sans]{dsfont}

\usepackage{graphicx}

\numberwithin{equation}{section}

\theoremstyle{plain}
\newtheorem{theorem}{Theorem}[section]
\newtheorem{proposition}[theorem]{Proposition}
\newtheorem{lemma}[theorem]{Lemma}

\newtheorem{prop}[theorem]{Proposition}

\newtheorem{corollary}[theorem]{Corollary}
\newtheorem{definition}[theorem]{Definition}

\theoremstyle{definition}
\newtheorem{remarkbase}[theorem]{Remark}
\newenvironment{remark}{\pushQED{\qed} \remarkbase}{\popQED\endremarkbase}

\newcommand{\1}{\mathcal{X}}
\newcommand{\vt}{\vartheta}
\newcommand{\ve}{\varepsilon}

\newcommand{\N}{{\mathbb N}}
\newcommand{\R}{{\mathbb R}}

\def\ba{\begin{aligned}}
\def\ea{\end{aligned}}
\def\beginm{\begin{multline}}
\def\endm{\end{multline}}

\title{Nonnegative controllability for a class of nonlinear degenerate parabolic equations\\ with application to climate science}

\author[G. Floridia]{Giuseppe Floridia} 
\address{Giuseppe Floridia\newline
Universit\`a Mediterranea di Reggio Calabria,\newline
 Dipartimento Patrimonio, Architettura, Urbanistica,\newline
Via dell'Universit\`a 25,
 Reggio Calabria, 89124, Italy}
\email{floridia.giuseppe@icloud.com}

\begin{document}
\markboth{G. Floridia}{Nonnegative controllability for a class\\ of nonlinear degenerate parabolic equations}
\maketitle
\begin{abstract}
\noindent Let us consider a nonlinear degenerate reaction-diffusion equation with application to climate science. After proving that the solution remains nonnegative at any time, when the initial state is nonnegative, we prove the approximate controllability between nonnegative states at any time via multiplicative controls, that is, using as control the reaction coefficient.
%
%
\end{abstract}
%

\begin{small}
\end{small}



\date{} 
\pagestyle{plain}

\section{Introduction} 
Let us consider a general function $a\in C([-1,1])\cap C^1(-1,1)$ such that $a
$ is strictly positive on $(-1,1)$ and  
$a(\pm1)\!=\!0,$ such as $(1-x^2)^\eta$ with $\eta>0.$
In this paper, we will study the following one-dimensional semilinear reaction-diffusion equation $$u_t-(a(x) u_x)_x =\alpha(x,t)u+ f(x,t,u),\;\;\; 
(x,t)\in Q_T :=(-1,1)\times(0,T),\;T>0,$$ 
where $\alpha$ is a bounded function on $Q_T$ and $f(\cdot,\cdot,u)$ is a suitable non-linearity that will be defined below. 
The above semilinear equation
 is a degenerate parabolic equation 
 since the diffusion coefficient 
vanishes at the boundary points of $[-1,1].$ 

Our interest in this kind of degenerate reaction-diffusion equations is motivated by its applications to the \-energy balance models in climate science, see e.g. the Budyko-Sellers model that is obtained from the above class of degenerate equations, 
in the particular case 
$a(x)=1-x^2.$ We devote the entire Section \ref{climate} of this paper to the presentation of these applications of degenerate equations to climate science.

In our mathematical study we need to distinguish between two classes of degenerate problems: weakly degenerate problems $(WDeg)$ (see \cite{CF2} and \cite{F2}) when the  degenerate diffusion coefficient is such that 
$\frac{1}{a}\in L^1(-1,1)$ (e.g. $a(x)=\sqrt{1-x^2}$), and
strongly degenerate problems $(SDeg)$ (see \cite{CFproceedings1} and \cite{F1}) when $\frac{1}{a}\not\in L^1(-1,1)$ (if $a \in C^1([-1,1])$ follows $\frac{1}{a}\not\in L^1(-1,1),$ e.g. $a(x)=1-x^2$).
It is well-known (see, e.g., \cite{ACF}) that, in the $(WDeg)$ case, all functions in the domain of the corresponding differential operator possess a trace on the boundary,
in spite of the fact that the operator degenerates at such points. Thus, in the $(WDeg)$ case we can consider the general Robin type boundary conditions, in a similar way to the uniformly parabolic case. Conversely, in the harder $(SDeg)$ case, one is limited to only the weighted Neumann type boundary conditions.

\noindent Such a preamble allows us to justify the following general problem formulation.

\subsection{Problem formulation}\label{Pbfor}
Let us introduce the following semilinear degenerate parabolic Cauchy problem 
\begin{equation}
\label{PS}
\left\{\begin{array}{l}
\displaystyle{u_t-(a(x) u_x)_x =\alpha(x,t)u+ f(x,t,u)\;\,\quad \mbox{ in } \; Q_T \,:=\,(-1,1)\times(0,T)}\\ [2.5ex]
\displaystyle{
\begin{cases}
\begin{cases}
\beta_0 u(-1,t)+\beta_1 a(-1)u_x(-1,t)= 0  \;\;\;t\in (0,T)\, \\
\qquad\qquad\qquad\qquad\qquad\qquad\qquad\qquad\qquad\quad\qquad(\mbox{for }\, WDeg)\\
\gamma_0\, u(1,t)\,+\,\gamma_1\, a(1)\,u_x(1,t)= 0
 \qquad\;\: t\in (0,T)\,
\end{cases}
\\
\quad a(x)u_x(x,t)|_{x=\pm 1} = 0\,\,\qquad\qquad\qquad\, t\in(0,T)\,\quad(\mbox{for }\, SDeg)
\end{cases}
}\\ [2.5ex]
\displaystyle{u(x,0)=u_0 (x) \in L^2(-1,1)\, 
}
~,
\end{array}\right.
\end{equation}
where 
the 
reaction coefficient $\alpha(x,t)\in L^\infty(Q_T)$ will represent the \textit{multiplicative control} (that is the variable function through which we can act on the system), and $\alpha$ is chosen, in this paper, as a piecewise static function (in the sense of Definition \ref{piece}). Throughout the paper
we always consider the problem \eqref{PS} under the following assumtions:
 \begin{enumerate}
\item[(SL)] $f:Q_T\times\R\rightarrow \R$ is such that
\begin{itemize}
\item $(x,t,u)\longmapsto f(x,t,u)$ is a Carath\'{e}odory 
 function on $Q_T\times\R,$ that is
\begin{itemize}
\item[$\star$] $(x,t)\longmapsto f(x,t,u)$ is measurable, for every $u\in\R,$ 
\item[$\star$] $u\longmapsto f(x,t,u)$ is a continuous function, for a.e. $(x,t)\in Q_T;$
\end{itemize}
\item
$t\longmapsto f(x,t,u)$ is locally absolutely continuous for a.e. $x\in (-1,1),$ for every $u\in\R,$\, and 
\begin{equation*}
  f_t(x,t,u)\,u\geq-\nu\, 
  u^2\,,\;  \text{ for a.e. } t\in (0,T);
\end{equation*}
\item there exist constants $\delta_*\geq 0, \vartheta\in[1,\vartheta_{sup}),$ $\vartheta_{sup}\in\{3,4\},$  
and $\nu
\geq0$
 such that, \mbox{for a.e.} $(x,t)\in Q_T, \forall u,v\in \R,$ we have
\begin{equation}\label{Superlinearit}
|f(x,t,u)|\leq\delta_*\,|u|^\vartheta, 
\end{equation}
\begin{equation}\label{fsigni}
-\nu
\big(1+|u|^{\vartheta-1}+|v|^{\vartheta-1}\big)(u-v)^2\leq \big(f(x,t,u)-f(x,t,v)\big)(u-v)\leq\nu (u-v)^2,
\end{equation}
\end{itemize}
   \item[(Deg)] $a \in C([-1,1])\cap C^1(-1,1)$ is such that
    $$a(x)>0, \,\, \forall \, x \in (-1,1),\quad a(-1)=a(1)=0,$$
    then, we consider the following two alternative cases:
     \begin{itemize}
\item[$(WDeg)$]
 if
 $\;\;\displaystyle\frac{1}{a}\in L^1(-1,1),$ then $\vartheta_{sup}=4$ and in \eqref{PS} let us consider the 
 Robin boundary conditions, where $\beta_0,\beta_1,\gamma_0,\gamma_1\in\R,\;\beta_0^2+\beta_1^2>0, \;\gamma_0^2+\gamma_1^2>0,$ satisfy the sign condition:
    \;\; $\beta_0\beta_1 \leq 0 \;\,\mbox{ and }\;\, \gamma_0\gamma_1 \geq 0;$
%
\item[$(SDeg)$] if $\;\,\displaystyle\frac{1}{a}\not\in L^1(-1,1)$
   and the function $\xi_a(x):=\displaystyle\int_0^x \frac{1}{a(s)}ds 
  \in L^{q_\vartheta}(-1,1),$
  where
    $ q_\vartheta=\max\Big\{\frac{1+\vartheta}{3-\vartheta}, 2\vartheta-1\Big\},$ $\vartheta_{\sup}=3,$ then in \eqref{PS} let us consider the weighted Neumann boundary conditions.
  \end{itemize}
\end{enumerate}

To better clarify the kind of multiplicative controls used, we recall the definition of piecewise static function.
\begin{definition}\label{piece}
We say that a 
function $\alpha\in L^\infty(Q_T)$ is {\it piecewise static} (or a {\it simple function} with respect to the variable {\it t}), if there exist $m\in\N,$ $\alpha_k(x)\in L^\infty(-1,1)$ and $t_k\in [0,T], \,t_{k-1}<t_k,\, k=1,\dots,m$ with $t_0=0 \mbox{ and } t_m=T,$ 
such that 
\begin{equation}\label{alpha_piece}
\alpha(x,t)=\alpha_1(x)\1_{[t_{0},t_1]}(t)+\sum_{k=2}^m \alpha_k(x){\1}_{(t_{k-1},t_k]}(t),
\end{equation}
 where ${\1}_{[t_{0},t_1]}\,  \mbox{  and  }  \,{\1}_{(t_{k-1},t_k]}$ are the indicator function of $[t_{0},t_1]$ and $(t_{k-1},t_k]$, respectively.
Sometime, for clarity purposes, we will call the function $\alpha$ in \eqref{alpha_piece} a {m-steps} piecewise static function
\end{definition}


\subsection{Main results
}\label{C}

In this paper we study the 
controllability of
 $\eqref{PS}$  using \textit{multiplicative controls}, that is the reaction coefficients $\alpha(x,t).$ 
 
%
First, we find that the following general nonnegative result holds also for the degenerate PDE of system \eqref{PS}. That is, if the initial state is nonnegative the corresponding strong solution to \eqref{PS} remains nonnegative  at any moment of time.\\
For the notion of strict/strong solutions of the nonlinear degenerate problem \eqref{PS} see Section \ref{Wp}.
 The following result is classic only for the uniformly parabolic (non degenerate) case.
\begin{proposition}\label{NN}
Let $u_0 \in L^2(-1,1)$ such that $\,u_0(x)\geq 0 \,\mbox{ a.e. } x \in (-1,1).$ 
Let $u$ 
 be the corresponding unique strong solution 
  to 
 $(\ref{PS}).$
 Then
$$u(x,t)\geq 0,\,\,\,\,\mbox{ for a.e. } (x,t)\in Q_T\,.
$$
\end{proposition}
A consequence of the result given in Proposition \ref{NN} is that
the solution to the system \eqref{PS} cannot be steered from a nonnegative initial state to any target state which is negative on a nonzero measure set in the space domain, regardless of the choice of the reaction coefficient $\alpha(x,t)$ as multiplicative control. 

Thus, in the following Theorem \ref{T1} we obtain an optimal goal, that is, we approximately control the system \eqref{PS} between nonnegative states via multiplicative controls at any time. 
Let us give the following definition.
%
%
%
\begin{definition}
The system \eqref{PS}  is said to be nonnegatively globally approximately controllable in $L^2(-1,1)$ at any 
time $T>0,$ by means of 
multiplicative controls $\alpha,$ 
if for any nonnegative 
$u_0,u^*\in L^2(-1,1)$ with $u_0\neq0,$ for every $\varepsilon>0$ 
there exists 
a piecewise static multiplicative control 
$\alpha=
\alpha (\varepsilon,u_0,u^*),\,\alpha
\in L^\infty(Q_T),$ such that for the correspon\-ding strong solution 
 $u(x,t)$ of $(\ref{PS})$
we obtain
$$\|u(\cdot,T)-u^*\|_{L^2(-1,1)}< \varepsilon. 
$$
\end{definition}

Now we can state the main controllability result.
\begin{theorem}\label{T1}
The nonlinear degenerate system \eqref{PS}
 is nonnegatively globally approximately controllable in $L^2(-1,1)$ at any 
 time $T>0,$ by means of 2-steps piecewise static multiplicative controls. 
\end{theorem}



\subsection{Outline of the paper} 
The proofs of the main results are given in Section \ref{main results}. In Section \ref{Wp} we recall the well-posedness of \eqref{PS}, in particular we introduce the notions of strict and strong solutions, and we give some useful estimates and properties for this kind of degenerate PDEs, that we use in Section \ref{main results}. The proofs of the existence and uniqueness results for strict and strong solutions are contained, in appendix, in Section \ref{AppA}.
In Section \ref{climate}
we present
some 
 motivations for studying degenerate parabolic problems with the above structure, in particular we introduce the Budyko-Sellers model, an energy balance model in climate science.
We complete this introduction with Section \ref{art} where we present the state of the art in both multiplicative controllability and degenerate reaction-diffusion equations.
\subsection{State of the art in multiplicative controllability and degenerate pa\-rabolic equations 
}\label{art}

Control theory appeared in the second part of last
century in the context of li\-near ordinary differential equations and was motivated by several engineering, Life sciences and economics applications.
Then, it  was extended to various linear partial differential equations (PDEs) governed by {additive} locally distributed controls (see \cite{ACF}, \cite{BFH}, \cite{CMVams16}, 
\cite{D1}, \cite{FattRuss}, \cite{FCZ} and \cite{PZ}), 
 or boundary controls. 
Methodologically-speaking, these kinds of controllability results for PDEs are typically obtained using the {linear} duality pairing technique between the control-to-state mapping at hand and its dual observation map (see the Hilbert Uniqueness Method - HUM - introduced in 1988 by J.L. Lions), sometimes using the Carleman estimates tool (see, e.g., \cite{ACF}, \cite{CFGY} and \cite{CFY1}).
If the above map is nonlinear, as it happens in our case for the multiplicative controllability, in general 
the aforementioned approach does not apply.

 From the point of view of applications, the approach based on multiplicative controls seems more realistic than the other kinds of controllability, since additive and boundary controls don't model in a realistic way the problems that involve inputs with high energy levels; such as energy balance models in climate science (see Se\-ction \ref{climate}), chemical reactions controlled by catalysts, nuclear chain reactions, smart materials, social science, ecological population dynamic (see \cite{TZZ}) and biome\-dical models. 
An important class of biomedical reaction-diffusion problems consists in the models of tumor growth 
(see, e.g., Section 7 {\it \lq\lq Control problems''} of the survey paper \cite{BP00} by Bellomo and Preziosi). As regards degenerate reaction-diffusion equations there are also interesting models in population genetics, in particular we recall the Fleming-Viot model (see Epstein's and Mazzeo's book \cite{EM}).


The above considerations motivate our investigation of the multiplicative controllability. As regards the topic of multiplicative controllability of PDEs we recall the pioneering paper 
\cite{BS} 
by Ball, Marsden, and Slemrod, and in the framework of the Schr\"odinger equation we especially mention \cite{BL} by Beauchard and Laurent, and \cite{CGM} by Coron, Gagnon and Moranceu.
 As regards parabolic and hyperbolic equations we focus on some results by Khapalov contained in the book \cite{KB}, and in the references therein. 

The main results of this paper deal with approximate multiplicative controllability of semilinear degenerate reaction-diffusion equations. This study is motivated by its applications (see in Section \ref{climate} its applications to an energy balance model in climate science: the Budyko-Sellers model) and also by the classical results
that hold for the corresponding 
non-degenerate reaction-diffusion equations, 
 governed via the coefficient of the reaction term (multiplicative control). 
 For the above class of uniformly parabolic  
 equations there are some important obstructions to multiplicative controllability due to the strong maximum principle (see the seminal papers by J.I. Diaz \cite{Dobstr} and \cite{D}, and also the papers \cite{CFK} and \cite{KB}), that implies the well-known nonnegative constraint. In this paper in Proposition \ref{NN}  we extend 
 to the semilinear degenerate system \eqref{PS} the above nonnegative constraint, that motivates our inve\-stigations regarding the nonnegative controllability for the semilinear degenerate system \eqref{PS} with general weighted Robin/Neumann boundary conditions.

Regarding the nonnegative controllability for reaction-diffusion equations, first, Khapalov in 
 \cite{KB} obtains the nonnegative approximate controllability in large time of the 
one dimensional 
 heat equation 
  via multiplicative controls. Thus,
%
the author and Cannarsa considered the linear degenerate problem associated with $(\ref{PS}),$ 
both in
 the {weakly degenerate} 
  $(WDeg)$ case,  in \cite{CF2},  
 and in the {strongly dege\-nerate} $(SDeg)$ case,
 in \cite{CFproceedings1}. 
Then, the author  in \cite{F1} investigates semilinear strongly degenerate problems. 
This paper can be seen as the final step of the study started in 
 \cite{CFproceedings1}, \cite{CF2} and \cite{F1},  
where the global nonnegative approximate controllability was obtained 
in large time.
Indeed, 
in this paper we introduce a new 
proof, 
that permits us to obtain the nonnegative controllability in
 {\it arbitrary small time} and consequently at any time, instead of {\it large time}. 
Moreover, the proof, contained in 
 \cite{F1}, of the nonnegative controllability in large time for the $(SDeg)$ case has the further obstruction that permitted to treat only superlinear growth, with respect to $u,$ of the nonlinearity function $f(x,t,u).$ While the new proof, adopted in this paper, 
permits us to control also linear growth of $f,$ with respect to $u.$ 
 
Finally, we mention some recent papers about the approximate multiplicative controllability for reaction-diffusion equations between sign-changing states:
 in \cite{CFK} by the author with Cannarsa and Khapalov regarding a semilinear uniformly parabolic system, 
and \cite{FNT} by the author with Nitsch and Trombetti, concerning degenerate parabolic equations.  Furthermore, 
some interesting contributions about exact controllability issues for evolution equations via bilinear controls have recently appeared, in particular we mention \cite{ACU} by Alabau-Boussouira, Cannarsa and Urbani, and \cite{DL} by Duprez and Lissy.

To round off the discussion regarding the multiplicative controllability, we note that recently there has been an increasing interest in these topics, so many authors are starting to extend the above results from reaction-diffusion equations to other operators. In \cite{V}, Vancostenoble proved a nonnegative controllability result in large time for a linear parabolic equation with singular potential, following the approach of \cite{CFproceedings1} and \cite{CF2}.
 An interesting work in progress, using 
 the technique of this paper,
 consists of approaching the problem of the approximate controllability via multiplicative control of nonlocal operators, e.g. the fractional heat equation studied in \cite{BWZ} by Biccari, Warma and Zuazua. Other interesting open problems are suggested by the papers \cite{FR} and \cite{KCPF}.

\section{Well-Posedness}\label{Wp}
The well-posedness of the $(SDeg)$ problem in \eqref{PS} 
is introduced in \cite{F1}, while the well-posedness of the $(WDeg)$ problem in \eqref{PS} 
is presented in \cite{F2}. 
In order to study the well-posedness of the degenerate problem $(\ref{PS})$, it is necessary to introduce in Section \ref{Wp1} the weighted Sobolev spaces $H^1_a(-1,1)$ and $H^2_a(-1,1),$ and their main properties. Finally, in Section \ref{WD} we introduce the notions of strict and strong solutions for this class of semilinear degenerate problems, and we give the corresponding existence and uniqueness results, that are proved in appendix in Section \ref{AppA}. 
\subsection{Weighted Sobolev spaces}\label{Wp1}
Let $a \in C([-1,1])\cap C^1(-1,1)$ such that the assumptions $(Deg)$ holds, we define the following spaces:
\begin{itemize}
 \item 
$
 H^1_a(-1,1)\!\!=\!\!\begin{cases}
\{u\in L^2(-1,1)\cap AC([-1,1])
|\sqrt{a}\,u_x\in L^2(-1,1)\}\:\mbox{for}\, (WDeg)\\
\{u\in L^2(-1,1)\cap AC_{loc}(-1,1)
|\sqrt{a}\,u_x\in L^2(-1,1)\}\;\mbox{for} (SDeg),
\end{cases}
\!\!\!\!\!
$\\
where $AC([-1,1])$ denotes the space of the absolutely continuous functions on $[-1,1],$ and $AC_{loc}(-1,1)$ 
denotes the space of the locally absolutely continuous functions on $(-1,1);$\\
\item $H^2_a(-1,1):=\{u\in H^1_a(-1,1)| \, au_x \in H^1 (-1,1)\}.$
\end{itemize}
See \cite{ACF} (and also \cite{CFproceedings1}, \cite{CF2}, \cite{F1} and \cite{F2}) for the main functional properties of these kinds of weighted Sobolev spaces, in particular we note that
$H^1_a(-1,1)$
and $H^2_a(-1,1)$ are Hilbert spaces with the natural scalar products induced,
respectively, by the following norms
$$\|u\|_{1,a}^2:=\|u\|_{L^2(-1,1)}^2+|u|_{1,a}^2 \;\;\;\mbox{ and }\;\;\; \|u\|_{2,a}^2:=\|u\|_{1,a}^2+\|(au_x)_x\|^2_{L^2(-1,1)},$$
where $|u|_{1,a}^2:=\|\sqrt{a}u_x\|_{L^2(-1,1)}^2$ is a seminorm.\\
We recall the following important remark.
\begin{remark}
The space $H^1_a(-1,1)$ is embedded in $L^\infty(-1,1)$ only in the weakly degenerate case (see \cite{ACF}, \cite{CFproceedings1}, \cite{CF2} 
and \cite{F1}).
\end{remark}
In \cite{CMV2}, see Proposition 2.1 (see also the Appendix of \cite{F1} and Lemma 2.5 in \cite{CMP}), the following proposition is obtained. 
 \begin{proposition}\label{caratH2}
In the $(SDeg)$ case, 
for every $u\in H^2_a(-1,1)$ we have
$$\lim_{x\rightarrow\pm1}a(x)u_x(x)
=0
\qquad\qquad\text{ and }\qquad\qquad
au\in H^1_0 (-1,1)\,.
$$ 
 \end{proposition}

\subsubsection{Some spectral properties} 

Let us define the operator $(A_0,D(A_0))$ in the following way
\begin{equation}
\label{DA0}
\left\{\begin{array}{l}
\displaystyle{D(A_0)\!\!=\!\!\begin{cases}\!\!
\left\{u\in H^2_a (-1,1)\bigg\rvert
\begin{cases}
\beta_0 u(-1)+\beta_1 a(-1)u_x(-1)= 0 
\\
\gamma_0\, u(1)\,+\,\gamma_1\, a(1)\,u_x(1)= 0 
\end{cases}
\!\!\!\!\!\!\! \right\}\;\mbox{for }\, (WDeg)\\
\;\;
H^2_a (-1,1)\qquad\qquad\qquad\qquad\qquad\qquad\qquad\qquad\;\;\quad\;\mbox{for }\, (SDeg)
\end{cases}
}\\ [4.8ex]
\displaystyle{\;\;\;\;A_0\,u=(au_x)_x\,,
\,\, \forall \,u \in D(A_0)}~.
\end{array}\right.
\end{equation}
\begin{remark}\label{DASdeg}
We note that in the $(SDeg)$ case,  for every $u\in D(A_0)$ Proposition \ref{caratH2} guarantees that $u$ satisfies the weighted Neumann boundary conditions.
\end{remark}

 In the case $\alpha\in L^\infty (-1,1),$ 
we define the operator $(A,D(A))$ as 
\begin{equation}\label{DA}
\left\{\begin{array}{l}
\displaystyle{
D(A)=D(A_0)
}\\ [2.5ex]
\displaystyle{A\,u=(au_x)_x +\alpha\,u, \,
\,\, \forall \,u \in D(A)}~.
\end{array}\right.
\end{equation}
We recall some spectral results obtained in \cite{CF2} for $(WDeg),$ and in \cite{CFproceedings1} for $(SDeg).$ Let us start with Proposition \ref{str cont}. 
\begin{proposition}\label{str cont}
$(A, D(A_0))$ is a closed, self-adjoint, dissipative operator with dense domain in $L^2 (-1,1)$.
Therefore, $A$ is the infinitesimal generator of a strongly continuous semigroup 
of bounded linear operators on $L^2(-1,1)$.
\end{proposition}
Proposition \ref{str cont} allows us to obtain the following.
\begin{proposition}\label{spectrum}
There exists an increasing sequence $\{\lambda_p\}_{p\in\N},$ with
$\lambda_p\longrightarrow +\infty$ \mbox{ as } $p\rightarrow\infty\,,$
such that the eigenvalues of the operator $(A_0,D(A_0))$
are given by $\{-\lambda_p\}_{p\in\N}$, and the corresponding eigenfunctions $\{\omega_p\}_{p\in\N}$ form a complete orthonormal system in $L^2(-1,1)$.
\end{proposition}
\begin{remark}\label{Legendre}\rm
In the case 
$a(x)=1-x^2$, that is in the case of the Budyko-Sellers model, 
 the orthonormal eigenfunctions of the operator $(A_0,D(A_0))$ are reduced to Legendre's polynomials (see the complete and interesting Section 5.6 in \cite{NK} and Remark 3.2 in \cite{F1}).
%
\end{remark}

\subsubsection{Spaces involving time: ${\mathcal{B}}(Q_T)$ and ${\mathcal{H}}(Q_T)$}

\noindent Given $T>0,$ let us define the Banach spaces:
$$\mathcal{B}(Q_T):=C([0,T];L^2(-1,1))\cap L^2(0,T;H^1_a (-1,1))$$
with the following norm
\begin{equation*}
\label{normaB}
 \|u\|^2_{\mathcal{B}(Q_T)}= \sup_{t\in
[0,T]}\|u(\cdot,t)\|_{L^2(-1,1)}^2
+2\int^T_{0}\int^1_{-1}a(x)u^2_x
dx\,dt\,,
\end{equation*}
and
$$
{\mathcal{H}}(Q_T):=L^{2}(0,T;D(A_0)
)\cap H^{1}(0,T;L^2(-1,1))\cap C([0,T];H^{1}_a(-1,1))
$$
with the following norm 
\begin{multline*}
\label{normaH}
\|u\|^2_{\mathcal{H}(Q_T)}=
 \sup_{
 [0,T]}\left(\|u\|_{L^2(-1,1)}
 ^2+\|\sqrt{a}u_x\|_{L^2(-1,1)}
 ^2\right)\\
 +\int_0^T\left(\|u_t\|_{L^2(-1,1)}
 ^2+\|(au_x)_x\|_{L^2(-1,1)}
 ^2\right)\,dt.
\end{multline*}

In \cite{F1}, for the $(SDeg)$ case, and in \cite{F2}, for the $(WDeg),$ the following embedding lemma for the space $\mathcal{H}(Q_T)$  is obtained.
\begin{lemma}\label{sob3}
Let $\vartheta\geq1.$ 
Then
$
{\mathcal{H}}(Q_T)\subset L^{2\vartheta}(Q_T)
$
\, and 
$$
\|u\|_{L^{2\vartheta}(Q_T)}
\leq c\,
 T^{\frac{1}{2\vartheta}}\,\|u\|_{{\mathcal{H}}(Q_T)},
$$
where $c$ is a positive constant. 
\end{lemma}
We note that Lemma \ref{sob3} holds in a more general setting than the assumptions $(SL)-(Deg),$ where $\vartheta\in[1,\vartheta_{sup}),$ with $\vartheta_{sup}\in\{3,4\}$.

\subsection{Existence and uniqueness of solutions of semilinear degenerate problems}\label{WD}

In order to study the well-posedness, we represent the semilinear problem $(\ref{PS})$ 
using the following abstract setting in the Hilbert space $L^2(-1,1)$
\begin{equation}\label{operatorNL}
 \left\{\begin{array}{l}
\displaystyle{u^\prime(t)=\left(A_0+\alpha(t) I\right)u(t)
+\phi(u)\,,\qquad  t>0 }\\ [2.5ex]
\displaystyle{u(0)=u_0 \in L^2(-1,1) 
}~,
\end{array}\right.
\end{equation}
where $A_0$ is the operator defined in \eqref{DA0}, $I$ is the identity operator 
 and, for every $u\in {\mathcal{B}(Q_T)},
$ 
the Nemytskii operator associated with the problem $(\ref{PS})$ is defined as
\begin{equation}\label{phimap}
\phi(u)(x,t):=f(x,t ,u(x,t)),\;\;\forall (x,t)\in Q_T.
\end{equation}
In \cite{F1}, for the $(SDeg)$ case, and in \cite{F2}, for the $(WDeg)$ case, the following proposition is proved.
\begin{proposition}\label{loclip}
Let 
$1\leq\vartheta<\vartheta_{sup},$ 
let $f:Q_T\times\R\rightarrow \R$ be a function that satisfies assumption $(SL),$ and let us assume that the conditions $(Deg)$ hold.\\ 
Then
 $\phi:{\mathcal{B}(Q_T)}\longrightarrow L^{1+\frac{1}{\vartheta}}(Q_T),$ defined in \eqref{phimap},
is a locally Lipschitz continuous map 
 and $\phi ({\mathcal{H}(Q_T)})\subseteq L^2(Q_T).$
 \end{proposition}

Proposition \ref{loclip} justifies the introduction of the following notions of {\it strict solutions} and {\it strong solutions}. 
Such notions are classical in PDE theory, see, for instance, the book \cite{BDDM1}, pp. 62-64 (see also \cite{F1} and \cite{F2}), and the pioneer paper \cite{Fr} by K. O. Friedrichs.

\subsubsection{Strict solutions 
}

In this section we give the notion of solutions of \eqref{PS} with initial state in $H^1_a(-1,1),$ that is, we give the definition of  \textit{strict solutions}, introduced in \cite{F1} for $(SDeg)$ and in \cite{F2} for $(WDeg)$.
\begin{definition}\label{strictsolution}
If 
$u_0\in H^1_a(-1,1),$ u is a \textit{strict solution} 
to \eqref{PS}, 
if $u\in{\mathcal{H}}(Q_T)$ and
$$
\left\{\begin{array}{l}
\displaystyle{u_t-(a(x) u_x)_x =\alpha(x,t)u+ f(x,t,u)\,\quad\quad\, \mbox{a.e.\;\; in } \; Q_T \,:=\,(-1,1)\times(0,T)}\\ [2.5ex]
\displaystyle{
\begin{cases}
\begin{cases}
\beta_0 u(-1,t)+\beta_1 a(-1)u_x(-1,t)= 0 \quad\;\; \mbox{a.e.\;\;} \;t\in (0,T)\, \\
\qquad\qquad\qquad\qquad\qquad\qquad\qquad\qquad\qquad\qquad\qquad\qquad\;\,(\mbox{for } WDeg)\\
\gamma_0\, u(1,t)\,+\,\gamma_1\, a(1)\,u_x(1,t)= 0
 \qquad\quad\,\mbox{a.e.\;\;} t\in (0,T)\,
\end{cases}
\\
\quad a(x)u_x(x,t)|_{x=\pm 1} = 0\,\,\qquad\qquad\qquad\;\;\,\mbox{a.e.\;\;}\, t\in(0,T)\;\;\,\quad(\mbox{for }\, SDeg)\;\;\;
\end{cases}
}\\ [2.5ex]
\displaystyle{u(x,0)=u_0 (x) \,\qquad\qquad\qquad\qquad\quad\qquad\qquad\;\;\;\; \,x\in(-1,1)}~.
\end{array}\right.
$$
%
\end{definition}
\begin{remark}
Since a strict solution $u$ is such that $u\in{\mathcal{H}}(Q_T)\subseteq L^2(0,T; D(A_0)
),$ we have $$u(\cdot,t)\in D(A_0),\;\;\; 
 \text{ for a.e. } t\in(0,T).$$ 
Thus, thanks to the definition of the operator $(A,D(A))$ given in \eqref{DA} and Remark  \ref{DASdeg},  we deduce that the associated 
 boundary conditions hold, for almost every 
 $t\in(0,T)$.
\end{remark}
In Section \ref{AppA} the following existence and uniqueness result for strict solutions is proved (see also the Appendix B of \cite{F1} for $(SDeg),$ and \cite{F2} for $(WDeg)$).
\begin{theorem}\label{exB}
For all $u_0\in H^1_a(-1,1)$ there exists a unique strict solution $u\in{{\mathcal{H}}(Q_T)}$ to \eqref{PS}.
\end{theorem}
The proof of Theorem  \ref{exB}  is given in Section \ref{AppA}.

\subsubsection{Strong solutions 
}
In this section we introduce the notion of solutions when the initial state belongs to $L^2(-1,1).$ These kinds of solutions are called {\it strong solutions} and is defined by approximation of a sequence of strict solutions. 
%
%
\begin{definition}\label{strong}
Let 
$u_0\in L^2(-1,1).$ We say that $u\in\mathcal{B}(Q_T)$ is a \textit{strong solution} of \eqref{PS}, if $u(\cdot,0)=u_0$ and 
there exists a 
 sequence $\{u_k\}_{k\in\N}$ in $\mathcal{H}(Q_T)$ such that, as $k\rightarrow\infty,$ $u_k\longrightarrow u \mbox{  in }  \mathcal{B}(Q_T)$ and, for every $k\in\N$, $u_k$ is the strict solution of 
the Cauchy problem
$$
\;\;\left\{\begin{array}{l}
\displaystyle{u_{kt}-(a(x) u_{kx})_x =\alpha(x,t)u_k+f(x,t,u_k)\,\quad \mbox{ a.e. \, in } \;Q_T:=\,(-1,1)\times(0,T) }\\ [2.5ex]
\displaystyle{
\begin{cases}
\begin{cases}
\beta_0 u_k(-1,t)+\beta_1 a(-1)u_{kx}(-1,t)= 0 \quad\;\; \mbox{\,a.e.\;\;} \;t\in (0,T)\, \\
\qquad\qquad\qquad\qquad\qquad\qquad\qquad\qquad\qquad\qquad\qquad\qquad\;\,(\mbox{for } WDeg)\\
\gamma_0\, u_k(1,t)\,+\,\gamma_1\, a(1)\,u_{kx}(1,t)= 0
 \qquad\quad\,\mbox{a.e.\;\;} t\in (0,T)\,
\end{cases}
\\
\quad a(x)u_{kx}(x,t)|_{x=\pm 1} = 0\,\,\qquad\qquad\qquad\;\;\,\mbox{\;\;a.e.}\;\; t\in(0,T)\quad(\mbox{for } SDeg)\;\;\;
\end{cases}
}\\ [2.5ex]
\end{array}\right.
$$
with initial datum $u_k(x,0).$
\end{definition}
\begin{remark}
Let us consider the sequence of strict solutions of Definition \ref{strong}, $\{u_k\}_{k\in\N}\subseteq\mathcal{H}(Q_T)$ such that, as $k\rightarrow\infty,$ $u_k\longrightarrow u \mbox{  in }  \mathcal{B}(Q_T).$ Thus, it follows that
$u_k(\cdot,0)\longrightarrow u_0 \mbox{  in }  L^{2}(-1,1),$ due to the definition of the $\mathcal{B}(Q_T)-$norm.
\end{remark}
%
%
For the proof of the main results, the next proposition, obtained in \cite{F1} for $(SDeg)$ and in \cite{F2} for $(WDeg),$ will be very useful. 
\begin{proposition}\label{uni}
Let $\alpha\in L^\infty(Q_T)$ a piecewise static function and let $u_0, v_0\in L^2(-1,1).$ 
Let $u,v$ be 
the corresponding 
strong solutions of  \eqref{PS}, with initial data $u_0, v_0$ respectively.
Then, we have
\begin{equation}\label{inedc}
\|u-v\|_{\mathcal{B}(Q_T)}\leq C_T\,\|u_0-v_0\|_{L^2(-1,1)},
\end{equation}
where $C_T=e^{(\nu+\|\alpha^+\|_{\infty})T}$ and $\alpha^+:=\max\{\alpha,0\}$ 
denotes the positive part of $\alpha$.
\end{proposition}
%
From Proposition \ref{uni} trivially follows Corollary \ref{uni neg}.
\begin{corollary}\label{uni neg}
Let $u_0\in L^2(-1,1), \,\alpha\in L^\infty(Q_T),$ $\alpha$ piecewise static function with $\alpha(x,t)\leq0$ in $Q_T$, and  
let $u$ be 
the corresponding 
strong solution of  \eqref{PS}. 
If $T\in(0,\frac{1}{4\nu}),$ we have
\begin{equation}\label{inedcx}
\|u\|_{C([0,T],L^2(-1,1))}\leq \sqrt{2}\,\|u_0\|_{L^2(-1,1)}. 
\end{equation}
\end{corollary}

Now, we can find the following existence and uniqueness result for strong solutions, given  in \cite{F1} for $(SDeg)$ and in \cite{F2} for $(WDeg)$). 
\begin{theorem}\label{strong th}
For each $u_0\in L^2(-1,1)$ and  $\alpha\in L^\infty(Q_T)$ piecewise static function there exists a unique strong solution 
to \eqref{PS}. 
\end{theorem}
The proof of Theorem \ref{strong th} is showed in appendix in Section \ref{AppA}.
\subsubsection{Some further estimates}
First, we recall the following Lemma \ref{esist glob0}, obtained in \cite{F1} (see Lemma B.2 in Appendix B), for the $(SDeg)$ case, and in \cite{F2}, for the $(WDeg)$ case, in the case of  static reaction $\alpha\in L^\infty(-1,1)$ (a similar argument to that used in Subsection \ref{strictth0} to prove Theorem \ref{exB} permits to extend the following lemma to the case of $\alpha\in L^\infty(Q_T)$ piecewise static function).
\begin{lemma}\label{esist glob0}
Let $\alpha\in L^\infty(Q_T)$ piecewise static function 
 and $u_0\in H^1_a(-1,1).$ 
The strict solution $u\in {\mathcal{H}}(Q_T)$ of system \eqref{PS}, under the assumptions $(SL)$ and $(Deg)$,
satisfies the following estimate
$$\|u\|_{{\mathcal{H}}(Q_{T})}\leq 
c\,e^{k T}\|u_0\|_{1,a},$$
where
$c=c(\|u_0\|_{1,a})
$ and \;
 $k$ are positive constants.
\end{lemma}
Using Lemma \ref{sob3}, Proposition \ref{loclip} and Lemma \ref{esist glob0}
we can obtain the following Proposition \ref{f in L2}. 
\begin{proposition}\label{f in L2}
Let 
$\alpha\in L^\infty(Q_T)$ piecewise static function and $u_0\in H^1_a(-1,1).$ 
Let
$u\in {\mathcal{H}}(Q_T)$ the strict solution 
of system \eqref{PS}, under the assumptions $(SL)$ and $(Deg)$. 
Then, the function
$(x,t)\longmapsto f(x,t,u(x,t))$
belongs to $L^2(Q_T)$ and the following estimate holds
$$
\|f(\cdot,\cdot,u)\|_{L^2(Q_T)}\leq Ce^{k \vt T}
\sqrt{T}\|u_{0}\|^{\vt}_{1,a
}\;, 
%
$$
where $C=C(\|u_{0}\|_{1,a}) \text{ and } k$ are positive constants.
\end{proposition}
\begin{proof} 
Applying the inequality \eqref{Superlinearit}, Lemma \ref{sob3} and Lemma \ref{esist glob0}, we obtain
$$
\int_0^T  \int_{-1}^1   f^2(x,t,u) dx dt\leq \delta_*^2 \int_0^T  \int_{-1}^1   |u|^{2\vt} dx dt\\
\leq 
c
T\|u\|^{2\vt}_{{{\mathcal{H}}(Q_T)}}
\leq 
 Ce^{2k \vt T}
T\|u_{0}\|^{2\vt}_{1,a
}, 
$$
where $c, 
\,C=C(\|u_{0}\|_{1,a}) \text{ and } k$ are positive constants. 
\end{proof}

\section{Proof of the main results}\label{main results}

In this section we prove the main results of this paper. In particular, in Section \ref{Nsec}
%
we prove Proposition \ref{NN}, that is, we find that the solution to \eqref{PS} remains nonnegative at any time when the initial state is nonnegative, regardless of the choice of the multiplicative control $\alpha(x,t)$. 
In Section \ref{control section},
we prove Theorem \ref{T1}, that is, the 
global approximate multiplicative controllability 
 between nonnegative states at any time.\\


In this section, we will use $\|\cdot\|,$ $\|\cdot\|_\infty$ and $\langle\cdot,\cdot\rangle$ 
 instead of
the norms $\|\cdot\|_{L^2(-1,1)}$ and $\|\cdot\|_{L^\infty(Q_T)},$ and the inner product $\langle\cdot,\cdot\rangle_{L^2(-1,1)}$, respectively. 

\subsection{Nonnegative solutions}\label{Nsec}

Before starting with the proof of Proposition \ref{NN} we give a regularity property of the positive and negative part of a function, that will be used in that proof.\\

\noindent Let 
$u:(-1,1)\rightarrow\R$ we consider the positive and negative part functions, respe\-ctively,
$$
u^+(x) := \max\left\{u(x),0\right\},\qquad u^-(x) := \max\left\{0,-u(x)\right\},\qquad
 x\in (-1,1)\,.
$$
Then we have the following equality
$$
u=u^+ -u^- \qquad\quad\quad \mbox{  in    }\,(-1,1)\,.
$$
For the functions $u^+$ and $u^-$  we give the following result of
regularity in weighted Sobolev's spaces, obtained as trivial consequence of a classical result for the usual Sobolev's spaces, that we can find, e.g., in the Appendix $A$ of \cite{KS}.
\begin{proposition}
\label{A.1}
Let 
$u\in H^{1}_a(-1,1),$
then
$
u^+,\,u^-\in H^{1}_a(-1,1).$ 
Moreover, we have 
\begin{equation*}
(u^+)_{x}=
\left\{\begin{array}{l}
{ u_{x}(x) \;\, \mbox{ if }\;\, u(x)>0 
}\\ [2.5ex]
{ 0 \;\;\;\quad\;\;\mbox{ if }\;\, u(x)\leq0
}~
\end{array}\right.
\quad\mbox{ and }\quad
(u^-)_{x}
=\left\{\begin{array}{l}
{ -u_{x}(x) \;\;\; \mbox{ if }\;\, u(x)<0 
}\\ [2.5ex]
{ 0 \,\,\qquad\,\,\,\;\; \mbox{ if }\;\, u(x)\geq0
 }~.
\end{array}\right.
\end{equation*}
\end{proposition}

Now, we can give the proof of Proposition \ref{NN}.
\begin{proof}{(of Proposition \ref{NN}).}\\
\underline{Case 1:} \textit{$u_0\in H^1_a(-1,1).$}\\
Firstly, let us prove Proposition \ref{NN} under the further assumption $u_0\in H^1_a(-1,1).$\\ So, we can note that the corresponding unique solution $u(x,t)$ is a {\it strict solution}, that is 
$$u\in{\mathcal{H}}(Q_T)=L^{2}(0,T;D(A_0)
)\cap H^{1}(0,T;L^2(-1,1))\cap C([0,T];H^{1}_a(-1,1)).$$ 
We denote with $u^+$ and $u^-$ the positive and negative part of $u$, respectively. Since $u=u^+-u^-,$ thus
it is sufficient to prove that $u^-(x,t)=0, \; \mbox{for a.e. in } Q_T\,.$
Multiplying by
$u^{-}$ 
 both sides of the equation in $(\ref{PS})$
and integrating on $(-1,1)$
we obtain
\begin{equation}
\label{4.2}
\int^1_{-1} u_{t} u^- dx= \int^1_{-1}\left[(a(x)u_{x})_x u^{-}+\alpha u
u^{-}+f(x,t,u) u^{-}\right]dx.
\end{equation}
We start with the estimation of the terms of the second member in \eqref{4.2}.\\
Integrating by parts and recalling that $u^-(\cdot,t)\in H^1_a (-1,1), \mbox{ for every } t\in (0,T),$ 
using Proposition \ref{A.1} we
deduce
\begin{align}\label{prpart}
\int^1_{-1} (a(x)u_{x})_x u^-\,dx &=[a(x)u_{x} u^-]^1_{-1} - \int^1_{-1}
a(x)u_{x}(u^-)_x\,dx\\
&=[a(x)u_{x} u^-]^1_{-1} + \int^1_{-1}
a(x)u^2_{x}\,dx\,.\nonumber
\end{align}
If $\beta_1\gamma_1\neq0,$ keeping in mind the boundary conditions, for $t\in(0,T)$ we have
\begin{multline}\label{bdterm}
[a(x)u_{x} u^-]^1_{-1}=a(1)u_{x}(1,t) u^-(1,t)-a(-1)u_{x}(-1,t) u^-(-1,t)\\
=-\frac{\gamma_0}{\gamma_1}(u^+(1,t)-u^-(1,t))u^-(1,t)
+\frac{\beta_0}{\beta_1}(u^+(-1,t)-u^-(-1,t))u^-(-1,t)\\
=\frac{\gamma_0}{\gamma_1}(u^-(1,t))^2
-\frac{\beta_0}{\beta_1}(u^-(-1,t))^2\geq0.
\end{multline}
Thus, including also the simple case $\beta_1\gamma_1=0,$ from \eqref{prpart} and \eqref{bdterm} we obtain
\begin{equation}\label{a(x)}
\int^1_{-1} (a(x)u_{x})_x u^-\,dx\geq0.
\end{equation}
We also have the following equality
\begin{equation}\label{alfa}
\int^1_{-1}\alpha u u^- dx = -\int^1_{-1}\alpha(u^-)^2 dx,
\end{equation}
moreover, using 
(\ref{fsigni})
we have
\begin{multline}\label{NL}
\int^1_{-1} f(x,t,u) u^-\,dx=\int^1_{-1} f(x,t,u^+-u^-) u^-\,dx
=\int^1_{-1} f(x,t,-u^-) u^-\,dx\\=-\int^1_{-1} f(x,t,-u^-) \left(-u^-\right)\,dx
\geq -\int^1_{-1}  \nu
\left(-u^-\right)^2\,dx=-\int^1_{-1}  \nu
\left(u^-\right)^2\,dx\,.
\end{multline}
We can compute the first member of the equation in \eqref{4.2} in the following way
$$\int^1_{-1}u_{t} u^- dx = \int^1_{-1} (u^+ - u^-)_t u^- dx =
-\int^1_{-1} (u^-)_t u^- dx = -\frac{1}{2} \frac{d}{dt} \int
(u^-)^2 dx\,.$$
Applying to (\ref{4.2})  the last equality and \eqref{a(x)}-\eqref{NL} we have
$$
\frac{1}{2}\frac{d}{dt}\int^1_{-1}(u^-)^2 dx\leq  \int^1_{-1}\left(\alpha(x,t)+\nu
\right) (u^-)^2
dx\leq \left(\|\alpha\|_\infty+\nu
\right)\int^1_{-1}(u^-)^2 dx,
$$
so by
Gronwall's Lemma since $u_0^-(x)\equiv0$ 
 we obtain
$$
\int^1_{-1}(u^-(x,t))^2 dx\leq
e^{2(\nu+\|\alpha\|_\infty) T}\,\,\int^1_{-1}(u^-(x,0))^2dx=0,\;\;\,\;
\forall t\in(0,T).
$$
Therefore,
\begin{equation}\label{preGW}
u^-(x,t)=0 \qquad \forall (x,t)\in Q_T,
\end{equation}
that proves Proposition \ref{NN} in the case $u_0\in H^1_a(-1,1).$\\

\noindent \underline{Case 2:} \textit{$u_0\in L^2(-1,1).$}\\ 
If $u_0\in L^2(-1,1),$
 $
u_0\geq0 \mbox{ a.e. } x\in(-1,1),$ there exists $\{u_{0k}\}_{k\in\N}\subseteq C^\infty([-1,1]),$ such that $u_{0k}\geq0$ 
on $(-1,1)$ for every $k\in\N,$ and $u_{0k}\longrightarrow u_0$ in $L^2(-1,1),$ as $k\rightarrow\infty.$
%
For every $k\in\N,$ we consider $u_k\in{\mathcal{H}}(Q_T)$ the strict solution to $(\ref{PS})$  with initial state $u_{0k}$. 
For the well-posedness there exists $u\in{\mathcal{B}}(Q_T)$ such that $u_k\longrightarrow u \mbox{ in }{\mathcal{B}}(Q_T),$ as $k\rightarrow\infty.$ This convergence implies that
there exists 
$\{u_{k_p}\}_{p\in\N}\subseteq\{u_{k}\}_{k\in\N}$ such that, as $p\rightarrow\infty,$
\begin{equation}
\label{estratta}
u_{k_p}(x,t)\longrightarrow u(x,t),\quad \mbox{ a.e. } (x,t)\in Q_T.
\end{equation}
For every $p\in\N,$ we can apply the Case 1 to the system \eqref{PS} with initial datum $u_{0k_p},$ so 
by \eqref{preGW} we deduce
$$
u_{k_p}(x,t)\geq0,\quad \mbox{ a.e. } (x,t)\in Q_T,
$$
thus from the convergence \eqref{estratta} the conclusion that follows is
$$u(x,t)\geq 0, \quad\mbox{ a.e. } (x,t)\in Q_T.$$
\end{proof}

\subsection{Nonnegative controllability}\label{control section}
%
In this section, let us give the proof of
Theorem \ref{T1}.

\begin{proof}{(Proof of Theorem \ref{T1}).}

\noindent Let us fix $\ve>0$. 
Since $u_{0},\,u^*\in L^2(-1,1),$ there exist
$u^\ve_{0},\,u^*_\ve \in C^1([-1,1])$ 
 such that 
 \begin{equation}\label{5.1}
u^\ve_0,\;u^*_\ve>0\;\: \text{ on }\,[-1,1],\;\;
\parallel u^*_\ve-u^*
 \parallel
 <\frac{\ve}{4} \;\text{  and  } \; \parallel u_{0}^\ve-u_{0} \parallel
<
 \frac{\sqrt{2}}{36 S_\ve e^{\nu}}\ve\,, 
\end{equation}
where $\nu$ is the nonnegative constant of assumptions $(SL)$ and
\begin{equation}\label{Sve}
\displaystyle S_\ve
:=
\max
_{x\in [-1,1]
}
\Big\{\frac{u^*_\ve (x)}{u_{0}^\ve (x)}\Big\}+1\,.
\end{equation} 
From  \eqref{5.1} and \eqref{Sve} follows
\begin{equation}\label{ass step}
\exists\, \eta^*>0\,:
\; \;\eta^*\leq
\frac{u^*_\ve
(x)}{S_\ve u^\ve_{0} (x)}
\leq 1, \;\;\qquad \forall x \in[-1,1]\,.
\end{equation} 

The strategy of the proof consists of two control actions: in the first step we steer the system from the initial state $u_0$ to the intermediate state $S_\ve u^\ve_0,$ then in the second step we drive the system from this to $u^*_\ve.$ In the second step, the condition \eqref{ass step} will be crucial, that is justify the choice of the intermediate state $S_\ve u^\ve_0.$

%

\noindent {\bf Step 1:} {\it Steering the system from $u_{0}
$ to $S_\ve u^\ve_{0}.$}
Let us choose the positive constant bilinear control 
\begin{equation*}
\alpha (x,t) = 	\alpha_1 := \frac{\log S_\ve}{T_1}>0, \;\;\;\;\; (x,t) \in (-1,1)\times(0, T_1), \;\;\;\text{ for some } T_1>0.
\end{equation*} 
Let us denote by $u^\ve(x,t)$ and $u(x,t)$  
the strict and strong solution 
of \eqref{PS} with initial state $u_0^\ve$ and $u_0$, respectively. 
So, keeping in mind the abstract formulation \eqref{operatorNL} for the problem \eqref{PS}, the Duhamel's principle and the Proposition \eqref{loclip}, the strict solution $u^\ve(x,t)$ is given, in terms of a Fourier series approach, by
\begin{equation}\label{w conv}
u^\ve (x,T_1)
= \; 
e^{\alpha_1 T_1}\,\sum_{p = 1}^\infty   e^{-\lambda_p T_1} 
\langle u^\ve_0,\omega_p\rangle
   \omega_p(x)
+ R_\ve(x,T_1)\,,
\end{equation}
with \;$\displaystyle R_\ve(x,T_1):= \sum_{p = 1}^\infty   \left[\int_0^{T_1} e^{(\alpha_1 -\lambda_p ) (T_1 -t)} 
\langle 
f (\cdot,t,u^\ve(\cdot, t)),\omega_p\rangle
dt
  \right]\omega_p(x)
,$\\ where
$\{-\lambda_p\}_{p\in\N}$
are the eigenvalues of the operator $(A_0,D(A_0)),$ defined in \eqref{DA0},
and $\{\omega_p\}_{p\in\N}$ are the corresponding eigenfunctions, 
 that form a complete orthonormal system in $L^2(-1,1),$ see Proposition \ref{spectrum}. 
 We recall that  the eigenvalues of the operator $(A,D(A)),$ with $Au=A_0u+\alpha_1u$ (defined in \eqref{DA}) are obtained from the eigenvalues of the operator $(A_0,D(A_0))$ by shift, that is we have $\{-\lambda_p+\alpha_1\}_{p\in\N},$ and the corresponding orthonormal system in $L^2(-1,1)$ of eigenfunctions is the same as $(A_0,D(A_0))$, that is 
 $\{\omega_p\}_{p\in\N}$.
 
By the strong continuity of the semigroup, see Proposition \ref{str cont}, 
we have that
$$
 \sum_{p = 1}^\infty   e^{-\lambda_p T_1} 
  \langle u^\ve_0,\omega_p\rangle \omega_p(x)\longrightarrow u^\ve_0
   \; \text{ in } L^2(-1,1) \;\;  \text{ as }\;\; T_1\rightarrow0.
 $$
 So, there exists a small time $\overline{T}_1\in(0,1)
$ such that 
 \begin{equation}\label{R}
\left\| 
S_\ve\sum_{p = 1}^\infty   e^{-\lambda_p T_1} 
\langle u^\ve_0,\omega_p\rangle\omega_p(\cdot)-S_\ve u^\ve_0(\cdot)\right\|<\frac{\ve 
}{8},\qquad \forall\: T_1\in(0,\overline{T}_1].
\end{equation}
Since $u^\ve$  is a strict solution, by Proposition \ref{f in L2} we have $f(\cdot,\cdot,u^\ve(\cdot,\cdot))\in L^2(Q_T),$ 
then using also H\"older's inequality and Parseval's identity we deduce
\begin{align}\label{H}
\|&R_\ve(x,T_1)\|^2
=
 \sum_{p = 1}^\infty \left| \int_0^{T_1} e^{(\alpha_1 -\lambda_p ) (T_1 -t)}
 \langle 
f (\cdot,t,u^\ve(\cdot, t)),\omega_p\rangle
 dt
\right|^2
\\
&\leq  \,
 \sum_{p = 1}^\infty \left( \int_0^{T_1}  e^{2(\alpha_1 -\lambda_p ) (T_1 -t)}dt
\right)\int_0^{T_1}\left|
\langle 
f (\cdot,t,u^\ve(\cdot, t)),\omega_p\rangle
\right|^2\!dt 
\nonumber\\
 &\leq
 e^{2\alpha_1T_1}T_1
\int_0^{T_1}\sum_{p = 1}^\infty \left|\langle 
f (\cdot,t,u^\ve(\cdot, t)),\omega_p\rangle
\right|^2\!dt 
\nonumber\\
&=S_\ve^2 
T_1 
\int_0^{T_1} \|f (\cdot,t,u^\ve(\cdot, t))\|^2
dt 
\leq 
 CS_\ve^2 e^{2k \vt T_1}
T_1^2\|u^\ve_{0}\|^{2\vt}_{1,a
}\,,\nonumber
\end{align}
where $C=C(\|u^\ve_{0}\|_{1,a
})$ and $k$ are the positive constants introduced in the statement of Proposition \ref{f in L2}.\\
Then there exists $T^*_1\in (0, \overline{T}_1]$ such that
\begin{equation}\label{small}
\sqrt{C}S_\ve e^{k \vt T_1}
T_1\|u^\ve_{0}\|^{\vt}_{1,a}
< \frac{\sqrt{2}}{36}\ve
,\;\;\qquad\forall\: T_1\in(0,T^*_1].
\end{equation}
Using Proposition \ref{uni}, by \eqref{w conv}-\eqref{small} and keeping in mind \eqref{5.1}, for every\\ $T_1\in(0,T^*_1],$
we obtain 
\begin{multline}\label{5.10}
\|u (\cdot, T_1) - S_\ve u^\ve_0 (\cdot)\|
\leq \|u (\cdot, T_1) - u^\ve (\cdot, T_1)\|
+\|u^\ve (\cdot, T_1) - S_\ve u^\ve_0 (\cdot)\|\\
\leq e^{(\nu+\|\alpha_1^+\|_\infty)T_1}\|u_0-u_0^\ve\|+ \left\| 
S_\ve\,\sum_{p = 1}^\infty   e^{-\lambda_p T_1} 
\langle u^\ve_0,\omega_p\rangle
   \omega_p(\cdot)-S_\ve u_0^\ve(\cdot)\right\|
+\| R_\ve(\cdot,T_1)\|\\
\leq e^{(\nu+\|\alpha_1^+\|_\infty)} \frac{\sqrt{2}}{36 S_\ve e^{\nu}}\ve+
\frac{\sqrt{2}}{36}\ve+\sqrt{C}S_\ve e^{k \vt T_1}
T_1\|u^\ve_{0}\|^{\vt}_{1,a
}<\frac{\sqrt{2}}{12}\ve, 
\end{multline}
where $\nu\geq0$ is given in assumption (SL).\\ 
Let us set 
\begin{equation}\label{sigma_0e}
\sigma_0^\ve(x):=u(x, T_1) - S_\ve u^\ve_0 (x),\end{equation}
we note that by \eqref{5.10} we have
\begin{equation}\label{sigma}
\|\sigma_0^\ve\|<\frac{\sqrt{2}}{12}\ve\,.
\end{equation}
{\bf Step 2:} {\it Steering the system from $S_\ve u^\ve_{0}+\sigma_0^\ve
$ to $u^*$ at $T\in(0,T^*],$ for some $T^*>0.$}
In this step let us restart at time $T_1$ from the initial state $S_\ve u^\ve_{0}+\sigma_0^\ve$ and our goal is to steer the system arbitrarily close to $u^*.$ 
Let us consider
\begin{equation}\label{ave}
\alpha_\ve (x) \; := \;  \left\{ \begin{array}{ll}
\log \left( \frac{u^*_\ve
(x)}{S_\ve u^\ve_{0} (x)}
 \right) \;\;\;\; &  {\rm for } \;\; x \neq  \pm1, 
   \\ 
0 \;\;\;\; &  {\rm for} \;\; x =  \pm1, 
\end{array}
\right.  
\end{equation} 
thus by \eqref{ass step} we deduce that $\alpha_\ve\in L^\infty(-1,1)$ and $\alpha_\ve(x)\leq0\;\:\text{ for a.e. } x\in [-1,1].$ So, 
there exists a sequence 
$\displaystyle \{ \alpha_{\ve j}\}_{j \in \N}\subset C^2( [-1, 1])$ such that 
\begin{equation}\label{a1}
\alpha_{\ve j} (x) \leq 0\;\; \forall x\in [-1,1],\; \;\alpha_{\ve j} (\pm1)=0, \;\alpha_{\ve j}\rightarrow \alpha_{\ve }\;{\rm in} \; L^2 (-1,1) \; {\rm as} \; j \rightarrow \infty,
\end{equation}
and the following further conditions hold
\begin{equation}\label{Ass alpha}
\displaystyle \lim_{x\rightarrow\pm1}\frac{\alpha_{\ve j}^\prime(x)}{a(x)}=0\;\; \text{ and }\;\;\lim_{x\rightarrow\pm1}\alpha_{\ve j}^\prime(x)a'(x)=0.
\end{equation}
%
%
Since, from \eqref{a1} we deduce 
$$
e^{\alpha_{\ve j} (x)}S_\ve u^\ve_{0} (x) 
 \; \longrightarrow  \;
e^{\alpha_\ve (x)} S_\ve u_{0}^\ve(x)
\; = \;u^*_\ve (x) \;\;{\rm in} \;\; L^2 (-1,1) \;\; {\rm as} \; j \rightarrow \infty,
$$
there exists $j^*\in\N$ such that for every $j\in\N,\:j\geq j^*$ we have
\begin{equation}\label{a2}
\|e^{\alpha_{\ve j}}S_\ve u^\ve_{0}-u_\ve^*\|<\frac{\ve}{12},\qquad\qquad  \forall j\in\N \quad \text{with} \quad\: j\geq j^*\,.
\end{equation}
Let us fix an arbitrary $j\in\N$ with $j\geq j^*,$ and let us choose as control
the following static multiplicative function
\begin{equation}\label{control2}
\alpha (x,t) \; := \; \frac{1}{T-T_1} \alpha_{\ve j} (x)\leq0     \;\;\;\:\; \forall (x,t) \in \widetilde{Q}_T:=(-1,1)\times(T_1,T),
\end{equation}
and we call $u^\sigma(x,t)$ the unique strong solution that solves the problem \eqref{PS} with the following changes:
\begin{itemize}
\item time interval $(T_1,T)$ instead of $(0,T);$ 
\item multiplicative control given by \eqref{control2};
\item initial condition $u^\sigma(x,T_1)=S_\ve u^\ve_0 (x)+\sigma_0^\ve(x).$
\end{itemize}
Let us also denote with $u(x,t)$ the unique strict solution of the following problem
 \begin{equation}\label{PSmod}
\left\{\begin{array}{l}
\displaystyle{u_t-(a(x) u_x)_x=
\frac{\alpha_{\ve j} (x)}{T-T_1} u+ f(x,t,u)\;\;\;\; \mbox{ in } \; \widetilde{Q}_T:=(-1,1)\times(T_1,T)
}\\ [2.5ex]
\displaystyle{ B.C.
%
%
}\\ [2.5ex]
\displaystyle{u(x,T_1)=S_\ve u^\ve_0 (x),\qquad\qquad\qquad\qquad\qquad\quad\,\;\;\;\: x\in (-1,1)
}
~.
\end{array}\right.
\end{equation}
%
For \text{a.e.} $x
\in (-1,1),$ from the equation 
$$u_t(\cdot,t)=\frac{\alpha_{\ve j}(\cdot)}{T-T_1}u(\cdot,t)+\left(\left(a(\cdot)u_{x}(\cdot,t)\right)_x +  
f(\cdot,t,u)\right)\;\;\;\;\quad t\in (T_1,T),$$
 by the classical variation constants technique, we can obtain a representation formula of the solution $u(x,t)$ of $\eqref{PSmod}$, that computed at time $T,$
 for every $x\in(-1,1),$ becomes 
\begin{equation*}
u(x,T) \; = \; 
e^{\alpha_{\ve j}(x)}S_\ve u_{0}^\ve(x) \; + \int_{T_1}^T e^{\alpha_{\ve j} (x) \frac{(T - \tau)}{T}} \big((a(x)u_{x})_x (x, \tau) + f (x,\tau, u(x,\tau))\big) d\tau\,.
\end{equation*}
Let us show that 
$u(\cdot, T) \longrightarrow u^*_\ve 
$ in $L^2 (-1,1),$ as $ T \rightarrow {T_1}^+.$\\
Since $ \alpha_{\ve j} (x)\leq 0$ let us note that by the above formula, using H\"older's inequality and \eqref{a2},  we deduce 
\begin{multline}\label{uT}
\!\!\!\!\!\!\|u (\cdot,T)-
u^*_\ve (\cdot)\|^2
\leq 2\|e^{\alpha_{\ve j}(x)}S_\ve u_{0}^\ve(x)-u^*_\ve\|^2\\
+2\int_{-1}^1 \left(\int_{T_1}^T e^{\alpha_{\ve j} (x)
\frac{(T - \tau)}{T}} ((a(x)u_{x})_x (x, \tau) + f (x,\tau,u(x,\tau))) d \tau \right)^2 \!\!\!dx\\ \;\leq  \frac{\ve^2}{72}\;+\:
\left(T-T_1\right) \parallel  \left(a(\cdot)u_{x}\right)_x  + f (\cdot,\cdot,u) \parallel^2_{L^2 (\widetilde{Q}_T)}.
\end{multline}
In the final Step. 3, as an appendix to this proof, we will prove that the norm at right-hand side of (\ref{uT}) is bounded 
 as $ T \rightarrow {T_1}^+$. Precisely, we will find 
 the following 
 \begin{equation}\label{bound}
  \parallel  \left(a(\cdot)u_{x}\right)_x  + f (\cdot,\cdot,u) \parallel_{L^2 (\widetilde{Q}_T)}\leq  K, \qquad \text{ for a.e.} \; 
    T\,\in (T_1, T_1+1),
 \end{equation}
 where $K=K(u_0^\ve,\,u_\ve^*,\,\|u_0^\ve\|_{1,a})$ is a positive constant.
 
 Thus, from \eqref{uT} and \eqref{bound} there exists $T_2\in (T_1, T_1+1)$ such that for every $T\in (T_1, T_2)$ we have
  \begin{equation}\label{bound2}
  \!\!\!\!\!\!\|u (\cdot,T)-
u^*_\ve (\cdot)\|<\frac{\ve}{6}.
   \end{equation}
 Then, using Corollary \ref{uni neg}, from \eqref{5.1}, \eqref{bound2} and \eqref{sigma}, there exists\\
  $T^*\in (T_1,\min\{T_2,T_1+\frac{1}{4\nu}\})$ such that for every $T\in(T_1,T^*]$ we obtain 
\begin{multline*}
\|u^\sigma(\cdot,T) -u^*(\cdot)\|\leq 
\|u^\sigma(\cdot,T)-u(\cdot,T)
\|
+ \|u(\cdot,T) -u^*_\ve(\cdot)\|+\|u^*_\ve-u^*\|\\
\leq
\sqrt{2}\|S_\ve u^\ve_{0}+\sigma_0^\ve-S_\ve u^\ve_{0}\|+\frac{\ve}{6}+\frac{\ve}{6}<\frac{\ve}{2},
\end{multline*}
from which 
follows the approximate controllability at any $T\in(0,T^*]$, since $T_1>0$ was arbitrarily small.
%
Moreover, if $T>T^*$ using the above argumentation 
we first obtain the ap\-proximate controllability at time $T^*.$ Then, we restart at time $T^*$ close to $u^*,$ and we stabilize the system into the neighborhood of $u^*,$ applying the above strategy 
$n$ times, for some $n\in\N,$ on $n$ small time interval by measure $\frac{T-T^*}{n},$ steering the system in every interval from a suitable approximation of $u^*$ to $u^*$.\\
{\bf Step 3:} {\it Evaluation of $\|\left(a(\cdot)u_{x}\right)_x  + f (\cdot,\cdot,u) \|^2_{L^2 (\widetilde{Q}_T)}$: Proof of the inequality \eqref{bound}}.\\ 
Multiplying by $\left(a(x)u_{x}\right)_x$ the equation in \eqref{PSmod}, 
integra\-ting 
over\\ $\widetilde{Q}_T=(-1,1)\times(T_1,T)$ and applying Young's 
inequality
we have
\begin{align}\label{a_x}
\parallel\left(a(\cdot)u_{x}\right)_x  &\parallel^2_{L^2 (\widetilde{Q} _T)}
\leq \int_{T_1}^T \int_{-1}^1 u_t \left(a(x)u_{x}\right)_x dx dt\\ &
- 
\frac{1}{T-T_1} \int_{T_1}^T \int_{-1}^1 \alpha_{\ve j}(x) u \left(a(x)u_{x}\right)_x dx dt \nonumber\\&+ \: 
\frac{1}{2}  \int_{T_1}^T  \int_{-1}^1   f^2(x,t,u) dx dt
+ \frac{1}{2}  \int_{T_1}^T  \int_{-1}^1  \left|\left(a(x)u_{x}\right)_x\right|^2 dx dt\,.\nonumber
\end{align}
Thus, by \eqref{a_x} using Proposition \eqref{f in L2} we obtain
\begin{multline}\label{a+f}
\|\left(a(\cdot)u_{x}\right)_x  + f (\cdot,\cdot,u)\|^2_{L^2 (\widetilde{Q}_T)}\leq
2\left( \| \left(a(\cdot)u_{x}\right)_x\|^2_{L^2 (\widetilde{Q}_T)}  + \|f (\cdot,\cdot,u) \|^2_{L^2 (\widetilde{Q}_T)}\right)\\
\leq 4
\int_{T_1}^T \int_{-1}^1 u_t \left(a(x)u_{x}\right)_x dx dt - 
\frac{4}{T-T_1} \int_{T_1}^T \int_{-1}^1 \alpha_{\ve j}(x) u \left(a(x)u_{x}\right)_x dx dt 
\\
+ 4C^2S_\ve^{2\vt}e^{2k \vt (T-T_1)}
\left(T-T_1\right)\|u^\ve_{0}\|^{2\vt}_{1,a
},
\end{multline}
where $C=C(\|u^\ve_{0}\|_{1,a})$ and $k$ are the positive constants given by Proposition \ref{f in L2}.
Let us estimate the first two terms of the right-hand side of \eqref{a_x}. Without loss of generality, let us consider the $(WDeg)$ problem with $\beta_0\gamma_0\neq0.$ 
Integrating by parts and using the sign condition $\beta_0\beta_1\leq0$ and $\gamma_0\gamma_1\geq0$ we have 
\begin{align}\label{deg2}
\;\;\;\;\;&\int_{T_1}^T \int_{-1}^1 u_t \left(a(x)u_{x}\right)_x dx dt \\
&=\int_{T_1}^T \left[u_t \left(a(x)u_{x}\right)\right]_{-1}^1 dt- \frac{1}{2}\int_{T_1}^T \int_{-1}^1 a(x)\left(u^2_{x}\right)_t dx dt
\nonumber\\
&\leq\frac{S_\ve}{2}\frac{\gamma_1}{\gamma_0}a^2(1)({u^\ve_{0x}}(1))^2-\frac{S_\ve}{2}\frac{\beta_1}{\beta_0}a^2(-1)(u^\ve_{0 x}(-1))^2+ \frac{S_\ve^2}{2} \int_{-1}^1 a(x)(u^\ve_{0 x})^2 dx dt\nonumber\\
&= c_1(S_\ve, u_0^\ve)+c_2(S_\ve)|u_0^\ve|^2_{1,a},\nonumber
\end{align}
where $c_1(S_\ve, u_0^\ve)\geq0$ and $c_2(S_\ve)>0$ are two constants. 
Let us note that in the $(SDeg)$ case or in the $(WDeg)$ problem with $\beta_0\gamma_0=0,$ we obtain a similar estimate, but in the third line of \eqref{deg2} at least one of the two boundary contributions is zero.\\
Furthermore, 
using \eqref{Ass alpha} and Proposition \ref{uni} 
we obtain
\begin{align}\label{deg3}
& \int_{T_1}^T \int_{-1}^1 \alpha_{\ve j}(x) u \left(a(x)u_{x}\right)_x dx dt \\
 &= 
 -\int_{T_1}^T \int_{-1}^1 \alpha_{\ve j}(x)a(x)u^2_{x} dx dt 
  -\frac{1}{2}\int_{T_1}^T \int_{-1}^1 \alpha'_{\ve j}(x)a(x)\left(u^2\right)_{x} dx dt \nonumber\\
&\geq -\frac{1}{2}\int_{T_1}^T \left[\alpha'_{\ve j}(x)a(x)u^2\right]_{-1}^1 dt+ \frac{1}{2}\int_{T_1}^T \int_{-1}^1 \left(\alpha''_{\ve j}(x)a(x)+\alpha'_{\ve j}(x)a'(x)\right)u^2 dx dt \nonumber\\
&\geq-\frac{1}{2}\sup_{x\in [-1,1]}\left|\alpha''_{\ve j}(x)a(x)+\alpha'_{\ve j}(x)a'(x)\right|  \int_{T_1}^T \int_{-1}^1u^2 dx dt\nonumber\\
 &\geq-(T-T_1) c(\alpha_{\ve j}',\alpha_{\ve j}'')e^{\nu (T-T_1)}S_\ve^2
\|u_{0}^\ve\|^2\,.\nonumber 
\end{align}
Finally, using \eqref{a+f}-\eqref{deg3}, we prove the inequality \eqref{bound}, that is for almost every \\$T\in(T_1,T_1+1)$ we have
\begin{multline*}
\|\left(a(\cdot)u_{x}\right)_x  + f (\cdot,\cdot,u)\|^2_{L^2 (\widetilde{Q}_T)}\leq
k_1(S_\ve, u_0^\ve)+k_2(S_\ve)|u_0^\ve|^2_{1,a}\\
+k_3(\nu,S_\ve,\alpha_{\ve j}',\alpha_{\ve j}'')
\|u_{0}^\ve\|^2+ k_4(S_\ve,\|u^\ve_{0}\|_{1,a})\|u^\ve_{0}\|^{2\vt}_{1,a
}\\
\leq k_1(S_\ve, u_0^\ve)+K_2(\nu,S_\ve,\alpha_{\ve j}',\alpha_{\ve j}'',\|u^\ve_{0}\|_{1,a})\|u^\ve_{0}\|^{2}_{1,a
},
\end{multline*}
where $k_1\geq0$ and $k_2,k_3,k_4,K_2>0$ are constants.
\end{proof}

\section{
Energy balance models in climate science}\label{climate}
Climate depends on various variables and parameters such as temperature,
humi\-dity, wind intensity, the effect of greenhouse gases, and so
on. It is also
affected by a complex set of interactions in the atmosphere,
oceans and continents, 
 that involve  physical, chemical,
geological and biological processes.\\ 
One of the first attempts to model the effects of the interaction between large ice masses and solar radiation on
climate is the one due, independently, to Budyko 
(see \cite{B2})
and Sellers (see \cite{S}). 
A complete treatment of the mathematical formulation of the Budyko-Sellers model has been obtained by J.I. Diaz and collaborators in \cite{D2}--\cite{DHT} and \cite{BCDT} (see also the interesting recent  monograph \cite{NK} on \lq\lq {\it Energy Balance Climate Models}'' 
by North and Kim, and 
some papers by P. Cannarsa and coauthors, see e.g. 
 \cite{CFproceedings1} and \cite{CMVams16} 
). \\
The Budyko-Sellers model is an {\it energy balance model}, which studies the role played by
continental and oceanic areas of ice on the evolution of the climate.
The effect of solar radiation on climate can be summarized in Figure 1. 

%

\begin{figure}
\label{EBM_2}
\begin{center}
\includegraphics[width=2.8in]{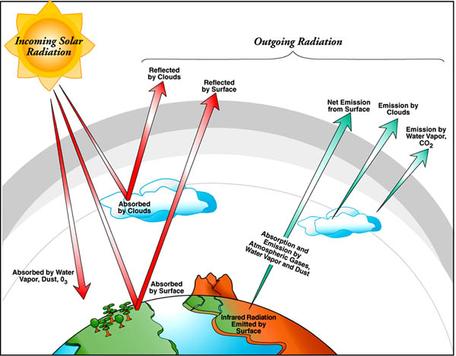}
\caption{Copyright by ASR}
\end{center}
\end{figure}
%
We have the following {\it energy balance}:\hspace{0.5cm}
$$\!\!\!\!\mbox{\textit{Heat variation}}=R_a-R_e+D,$$
where $ R_a$ is the \textit{absorbed
energy}, $ R_e$ is the \textit{emitted energy} and $ D$ is the \textit{diffusion part}.
If we represent the Earth by a compact two-dimensional manifold without boundary $\mathcal M,$ the general formulation of the Budyko-Sellers model is as follows
\begin{equation}\label{BSM}
c(X,t)u_t(X,t)-\Delta_{\mathcal M} u(X,t)= R_a(X,t,u)-R_e(
u),
\end{equation}
\noindent where $c(X,t)$ is a positive function (the heat capacity of the Earth), $u(X,t)$ is the annually (or seasonally) averaged Earth surface temperature, and $\Delta_{\mathcal M}$ is the classical Laplace-Beltrami operator. In order to simplify the equation \eqref{BSM}, in the following we can assume that the thermal capacity is $c\equiv 1.$
$R_e(u)$ denotes the Earth radiation, that is, the mean emitted energy flux, that depends on the amount of greenhouse gases, clouds and water vapor in the atmosphere and may be affected by anthropo-generated changes. In literature there are different empiric expressions of $R_e(u).$ In \cite{S}, Sellers proposes a Stefan-Boltzman type radiation law:
$$R_e(u)=\ve(u) u^4,$$
where $u$ is measured in Kelvin (and thus $u>0$), the positive function $\ve(u)=\sigma\left(1-m\tanh(\frac{19u^6}{10^6})\right)\!$ represents the emissivity, $\sigma$ is the emissivity constant and $m>0$ is the atmospheric opacity.
 In its place, in \cite{B2} Budyko considers a Newtonian linear type radiation, that is,
$R_e(u)=A+Bu,$
 with suitable $A\in\R,\,B>0,$ which is a linear approximation of the above law near the actual mean temperature of the Earth, $u=288,15 K \,(15^\circ C)$. \\
$R_a(X,t,u)$ denotes the fraction of the solar energy absorbed by the Earth and is assumed to be of the form
$$R_a(X,t,u)=QS(X,t)\beta(u),$$
in both the models. In the above relation,
$Q$ is the {\it Solar constant}, $S(X,t)
$ is the distribution of solar radiation over the Earth, in seasonal models (when the time scale is smaller) $S$ is a positive {\it\lq\lq almost periodic''} function in time (in particular, it is constant in time, $S=S(X),$ in annually averaged models, that is, when the time scale is long enough), and $\beta(u)$ is the {\it planetary coalbedo} representing the fraction absorbed according the average temperature ($\beta(u)\in [0,1]$) 
The \textit{coalbedo} function is equal to 1-\textit{albedo function}.
In climate science the albedo (see Figure 2) is more used and well-known than the coalbedo, and is the reflecting power of a surface. It is defined as the ratio of reflected radiation from the surface to incident radiation upon it. It may also be expressed as a percentage, and is measured on a scale from 0, for no reflecting power of a perfectly black surface, to 1, for perfect reflection of a white surface
.
The coalbedo is assumed to be a non-decreasing function of $u,$ that is, over ice-free zones (like oceans) the coalbedo is greater than over ice-covered regions. Denoted with $u_s=263,15K (
-10^\circ C)$ the critical value of the temperature at which ice becomes white (the \lq\lq {\it snow line}''), given two experimental values $a_i$ and $a_f,$ such that $0< a_i<a_f<1,$ in \cite{B2}
 Budyko 
proposes the following coalbedo function, discontinuity at $u_s,$ 
$$
 \beta(u)=\left\{\begin{array}{l}
\displaystyle{a_i, \qquad \text{over ice-covered}\qquad\{X\in {\mathcal M}:u(X,t)<u_s\},}\\ [2.5ex]
\displaystyle{a_f, \qquad \text{over ice-free}\qquad\;\:\quad\{X\in {\mathcal M}:u(X,t)>u_s\}}~.
\end{array}\right.
$$
Coversely, in \cite{S} Sellers proposes a more regular (at most Lipschitz continuous) function of $u.$ Indeed, Sellers represents $\beta(u)$ as a continuous piecewise linear function (beetween $a_i$ and $a_f$) with greatly increasing rate near $u=u_s,$ such that $\beta(u)=a_i,$ if $u(X,t)<u_s-\eta$ and $\beta(u)=a_f,$ if $u(X,t)>u_s+\eta,$ 
 for some small $\eta>0.$ 
 If we assume that $\mathcal M
$ is the unit sphere of $\R^3,$ the Laplace-Beltrami operator becames
$$\Delta_{\mathcal M}\,u=\frac1{\sin \phi}\Big\{\frac{\partial}{\partial \phi}\Big(\sin \phi \frac{\partial u}{\partial\phi}\Big)+\frac1{\sin \phi}\,\frac{\partial^2u}{\partial \lambda^2}\Big\},$$
where $\phi$ is the \textit{colatitude} and $\lambda$ is the \textit{longitude}.
\begin{figure}[pb]\label{Albedo}
\centerline{\includegraphics[width=1.3in]{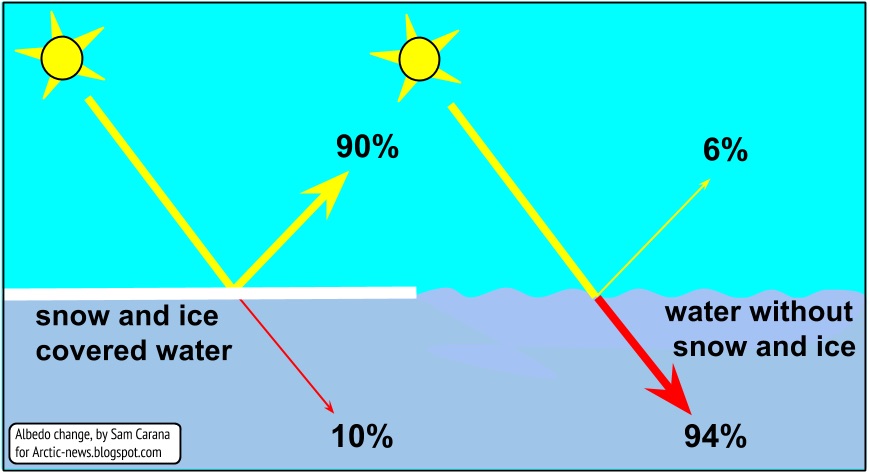}}
\vspace*{8pt}
\caption{Copyright by ABC Columbia}
\end{figure}
{\begin{figure}[pb]\label{mean}
\centerline{\includegraphics[width=1.3in]{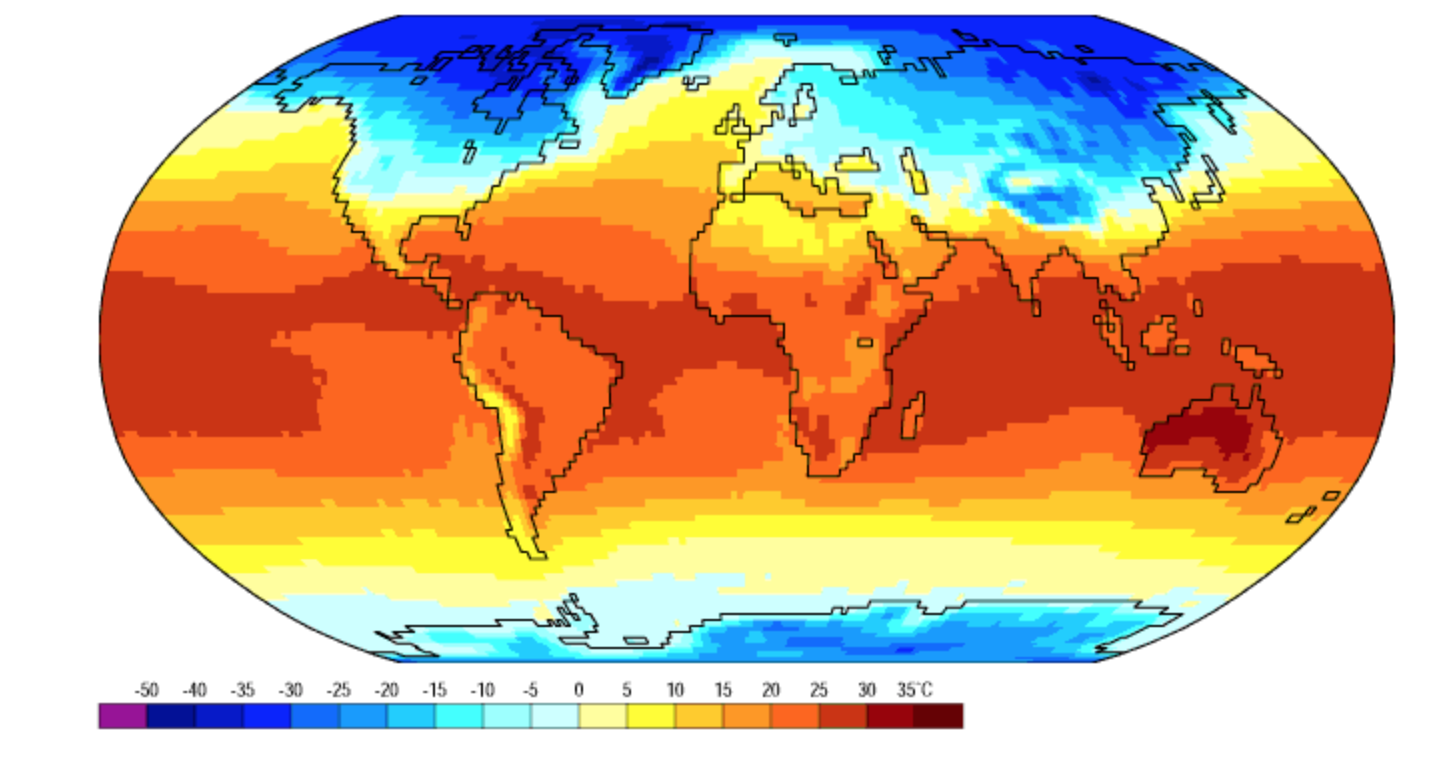}}
\vspace*{8pt}
\caption{Copyright by Edu-Arctic.eu}
\end{figure}}

Thus, if we take the average of the temperature at $ x=\cos\phi$ (see in Figure 3,
 that the distribution of the temperature at the same colatitude can be considered approximately uniform). 
In such a model, the sea level mean
zonally averaged temperature $u(x,t)$ on the Earth, where $t$ still
denotes time, 
 satisfies a \textit{Cauchy-Neumann} strongly degenerate problem, 
 in
the bounded domain $(-1,1),$ of the following type 
\begin{equation*}
  \left\{\begin{array}{l}
\displaystyle{u_t-\big((1-x^2)u_x\big)_x=\alpha(x,t)\,\beta(
u)+f(x,t,u), \qquad\,\, x\in (-1,1),}\\ [2.5ex]
\displaystyle{
\lim_{x\rightarrow \pm1}(1-x^2)u_{x}(x,t)
=0, \qquad\qquad\qquad\qquad\quad\quad\quad\,\;\, t\in(0,T)
}~.
\end{array}\right.
\end{equation*}
Then, the uniformly parabolic equation \eqref{BSM} has been transformed into a 1-D degenerate parabolic equation.
So, we have showen that our degenerate reaction-diffusion system \eqref{PS} reduces to the 1-D Budyko-Sellers model when $a(x)=1-x^2.$

%
\subsubsection*{Environmental aspects}
We remark that the Budyko-Sellers model studies the effect of solar radiation on climate, so it takes into consideration the influence of \lq\lq\textit{greenhouse gases}'' on climate. These cause \lq\lq\textit{global warming}'' which, consequently, provokes the increase in the average temperature of the Earth's atmosphere and of oceans. 
 This process consists of a warming of the Planet Earth by the action of greenhouse gases, compounds present in the air in a relatively low concentration (carbon dioxide, water vapour, methane, etc.). Greenhouse gases allow solar radiation to pass through the atmosphere while obstructing the passage towards space of a part of the infrared radiation from the Earth's surface and from the lower atmosphere.
The majority of climatologists believe that Earth's climate is destined to change, because human activities are altering atmosphere's chemical composition. In fact, the enormous anthropogenic emissions of greenhouse gases are causing an increase in the Earth's temperature, consequently, provoking profound changes in the Planetary climate.
One of the aims of this kind of research is to estimate the possibility of controlling the variation of the temperature over decades and centuries 
 and it proposes to provide a study of the possibility of slowing down \textit{global warming.}

\subsubsection*{Related open problems} Keeping in mind the meaning of the multiplicative control $\alpha$ in the climate framework, since in the main control result of this paper, Theorem \ref{T1}, the action must be realized over any latitude x in $[-1,1],$ it would be more realistic follow up in future papers the formulation that was already proposed by Von Neumann in 1955 (see the paper \cite{DVon} by J.I. Diaz for its comprehensive presentation), that is,  by using a localized control defined merely for some set of latitudes. From the multiplicative controllability point of view, that problem is hard but it is under research by J.I. Diaz and the author; one possible approach consists to follow some ideas introduced in \cite{FK}, in the case of uniformly parabolic equations.
 
%
%
%
%


%

%
%
\section{Appendix: The proofs of the existence and uniqueness results 
}\label{AppA}
In this appendix, in Section \ref{strictsec} we prove Theorem \ref{exB},
that is,
we show that, for all $\alpha\in L^\infty(Q_T)$ piecewise static functions, there exists a unique strict solution $u\in{\mathcal{H}(Q_T)}$ to $(\ref{PS})$, for all initial state $u_0\in H^1_a(-1,1)$. Thus, in Section \ref{strongth} we can prove Theorem \ref{strong}, that is, by an approximation argument we obtain the existence and uniqueness of the strong solution to \eqref{PS}, for all initial state $u_0\in L^2(-1,1)$.
\subsection{Existence and uniqueness of the strict solution to \eqref{PS}}\label{strictsec}
The aim of all Section \ref{strictsec} is to obtain the proof  of Theorem \ref{exB}, that is given in Subsection \ref{strictth0}.  
For this purpose, we will follow the following strategy:
\begin{itemize}
\item in Subsection \ref{MR} we present a {\it maximal regularity} result for abstract nonhomogeneous linear evolution equations in Hilbert spaces;
\item in Subsection \ref{strictth1} we introduce the notion of {\it mild solutions} and we give an existence and uniqueness result for mild solutions;
\item in Subsection \ref{local_strict} we prove the existence and uniqueness of strict solutions for static coefficient $\alpha\in L^\infty(-1,1);$
\item in Subsection \ref{strictth0}, finally we prove that if $u_0\in H^1_a(-1,1)$ then the mild solution is also a strict solution, for all $\alpha\in L^\infty(Q_T)$ piecewise static \- function. 
\end{itemize}


\subsubsection{A \textit{maximal regularity} result
for linear problems 
}\label{MR}

Let us consider
the following linear problem 
in the Hilbert space $L^2(-1,1)$
\begin{equation}\label{Ball}
 \left\{\begin{array}{l}
\displaystyle{u^\prime(t)=A\,u(t)\,+g(t),\qquad  t>0 }\\ [2.5ex]
\displaystyle{u(0)=u_0\, 
}~,
\end{array}\right.
\end{equation}
where $A$ is the operator in (\ref{DA}), $g\in L^1(0,T;L^2(-1,1)),\,u_0\in L^2(-1,1)$.\\ 
First, let us recall the notion of {\it \lq\lq weak solution''} introduced by J. Ball in \cite{Ba} for the linear problem $(\ref{Ball})$ (see, e.g., also \cite{CMP}, \cite{ACF}, \cite{F1} and  \cite{F2}).
\begin{definition} We define a \textit{weak solution} of the linear problem (\ref{Ball}) a function $u\in C([0,T];L^2(-1,1))$ such that for every $v\in D(A^*)$ ($A^*$ denotes the adjoint of A) the function $\langle u(t),v\rangle$ is absolutely continuous on $[0,T]$
and
$$\frac{d}{dt}\langle u(t),v\rangle=\langle u(t),A^*v\rangle\,+\langle g(t),v\rangle,$$
for almost all $t\in [0,T].$ 
\end{definition}
Then, we recall the following existence and uniqueness result obtained by J. Ball in \cite{Ba} (see also \cite{CFproceedings1}, \cite{CF2}, \cite{F1} and  \cite{F2}).
\begin{proposition}\label{Ball_mild}
For every $u_0 \in L^2(-1,1)$ there exists a unique weak solution $u$
of (\ref{Ball}), which is given by the following representation formula 
$$u(t)=e^{tA}u_0+\int_0^t e^{(t-s)A} g(s)\,ds,\; \;\;t\in[0,T].$$ 
\end{proposition}
Now, we are able to present Proposition \ref{MaxReg}, a \textit{maximal regularity} result that holds
in the Hilbert space $L^2(-1,1).$ Before giving the statement of Proposition \ref{MaxReg} we recall that 
by \textit{maximal regularity} for (\ref{Ball}) we mean that $u^\prime$ and $Au$ have the same regularity of $g$.
\begin{prop}\label{MaxReg}
Given $T>0$ and $g\in L^2(0,T;L^2(-1,1))
$\,
. 
For every 
$u_0 \in H^1_a(-1,1)$, there exists a unique
solution of (\ref{Ball})
$$u\in{\mathcal{H}}(Q_T)=L^{2}(0,T;D(A_0)
)\cap H^{1}(0,T;L^2(-1,1))\cap C([0,T];H^{1}_a(-1,1))\,. 
$$
Moreover, a positive constant $C_0(T)$ exists (nondecreasing in $T$), such that the following inequality holds
$$\|u\|_{{\mathcal{H}}(Q_T)}\leq C_0(T)\left[\|u_0\|_{1,a}+\|g\|_{L^{2}(Q_{T})}\right].$$
\end{prop}
\begin{proof}
It is a direct consequence of Theorem 3.1 in Section 3.6.3 of \cite{BDDM1}, pp. $79-82$ (see also \cite{CV} and \cite{F1}
), keeping in mind the following three crucial issues related to the abstract setting in Section 3.6.3 of \cite{BDDM1}: 
\begin{itemize}
\item $g\in L^2(0,T;X),$ where $X$ is the Hilbert space $L^2(-1,1)$;
\item $A$ is the infinitesimal generator of an {\it analytic} semigroup (see, e.g., \cite{CRV});
\item $u_0 \in H^1_a(-1,1),$ where $H^1_a(-1,1)$ is an interpolation space between the domain $D(A_0)$ and $L^2(-1,1).$
\end{itemize}
\end{proof}

\subsubsection{Existence and uniqueness of the mild solution to \eqref{PS}}\label{strictth1}
Before
introduce the notion of {\it mild solutions} in Definition \ref{mild}, in this section
we consider the following abstract representation of the semilinear problem $(\ref{PS})$ 
in the Hilbert space $L^2(-1,1)$
\begin{equation}\label{NLmild}
 \left\{\begin{array}{l}
\displaystyle{u^\prime(t)=A_0\,u(t)+\psi(t,u(t))\,,\qquad  t>0 }\\ [2.5ex]
\displaystyle{u(0)=u_0 \in L^2(-1,1) 
}~,
\end{array}\right.
\end{equation}
where $A_0$ is the operator defined in \eqref{DA0} 
 and, for every $u\in {\mathcal{B}(Q_T)},
$ 
\begin{equation}\label{psimap}
\psi(t,u)=:
\psi(x,t,u(x,t))=
\alpha(x,t)u(x,t)+f(x,t ,u(x,t)),\;\;\forall (x,t)\in Q_T,
\end{equation}
with $\alpha\in L^\infty(Q_T),$ piecewise static, given in \eqref{PS}.

We note that, since from \eqref{fsigni}, $\mbox{ for a.e. } (x,t)\in Q_T,\;\forall u,v\in \R,$ it follows that
$$
|f(x,t,u)-f(x,t,v)|
\leq\nu\left(1+|u|^{\vartheta-1}+|v|^{\vartheta-1}\right)|u-v|,
$$
we deduce
\begin{multline*}
|\psi(t,u)-\psi(t,v)|\leq |\alpha(x,t)u-\alpha(x,t)v|+|f(x,t,u)-f(x,t,v)|\\\leq \|\alpha\|_\infty |u-v|+\nu\left(1+|u|^{\vartheta-1}+|v|^{\vartheta-1}\right)|u-v|
.
\end{multline*}
Thus,
\begin{equation}\label{lip_psi}
|\psi(t,u)-\psi(t,v)|\leq L |u-v|, \;\;\;  \mbox{ for a.e. } \,(x,t)\in Q_T,\;\forall u,v\in L^2(-1,1),
\end{equation}
where $L=L(u,v)=\|\alpha\|_\infty+\nu\left(1+|u|^{\vartheta-1}+|v|^{\vartheta-1}\right)$ doesn't depend on t.
\begin{definition}\label{mild}
Let 
$u_0\in L^2(-1,1).$ We say that $u\in C([0,T];L^2(-1,1))
$ is a \textit{mild solution} of \eqref{PS}, if 
$u$ is a solution of the following integral equation
$$u(t)=e^{tA_0}u_0+\int_0^te^{A_0(t-s)}\psi(s,u(s))\,ds,\qquad t\in[0,T].$$
\end{definition}
The existence and the uniqueness of the mild solution of \eqref{PS} follows from the following Proposition \ref{mild_LY}, that is a consequence of a result contained in the Li and Yong's book \cite{LY}, see Proposition 5.3 in Chapter 2, Sec. 5, pp. 63--66. For the next proposition the following general assumptions on the function $\psi$ are introduced:
\begin{enumerate}
\item[($\psi$SM)] for each $\bar{u}\in L^2(-1,1),$ $\psi(\cdot,\bar{u}):[0,T]\longrightarrow L^2(-1,1)$ is {\it strongly mea\-surable}, that is there exists a sequence of {\it simple functions} (piecewise static fun\-ctions) $\psi_k(\cdot,\bar{u}):[0,T]\longrightarrow  L^2(-1,1)$ such that
$$\lim_{k\rightarrow\infty}|\psi_k(t,\bar{u})-\psi(t,\bar{u})|=0,\;\;\;\,\;\;\;\text{ a.e. }\;\;t\in[0,T];$$
\item[($\psi$L)] there exists a function $L(t)\in L^1(0,T)$ such that $$|\psi(t,u)-\psi(t,v)|\leq L(t) |u-v|, \;\;\;  
 \,\text{ a.e. } t\in[0,T],\;\forall u,v\in L^2(-1,1),$$
\qquad\qquad\qquad\qquad\qquad\qquad 
and
$$|\psi(t,0)|\leq L(t),\,\text{ a.e. } t\in[0,T].$$
\end{enumerate}

\begin{proposition}\label{mild_LY}
There exists a unique mild solution to 
\eqref{NLmild} 
under the assumptions ($\psi$SM) and ($\psi$L).
\end{proposition}

\subsubsection{Existence and uniqueness of strict solutions for static coefficient $\alpha\in L^\infty(-1,1)$
}\label{local_strict}
In this section
 we obtain the proof of Theorem \ref{exB}, proving that if the initial state belongs to $H^1_a(-1,1),$ then the {\it mild solution} is also {\it strict}.


\begin{lemma}\label{AppL1}
For every $M>0,$ there exists $T_M>0$ 
such that for all $\alpha\in L^\infty(-1,1),$ 
and all $u_0\in H^1_a(-1,1)$ with $\|u_0\|_{1,a}\leq M$ there is a unique strict solution $u\in{\mathcal{H}}(Q_{T_M})$ to $(\ref{PS}).$
\end{lemma}

\begin{proof}
Let us fix $M>0,$ $u_0\in H^1_a(-1,1)$ such that $\|u_0\|_{1,a}\leq M.$
Let $0<T\leq 1,$ 
we define
$${\mathcal{H}}_M(Q_{T}):=\{u\in {\mathcal{H}}(Q_{T})\,:\,\|u\|_{{\mathcal{H}}(Q_{T})}\leq 2C_0(1) M\},$$
 where $C_0(1)$ is the constant $C_0(T)$ (nondecreasing in $T$) defined in Proposition \ref{MaxReg} and valued in $T=1$.   Then, let us consider the map
$$\Phi :{\mathcal{H}}_M(Q_{T})\longrightarrow{\mathcal{H}}_M(Q_{T}),$$
defined by
\begin{equation*}\label{contr}
   \Phi(u)
   :=e^{tA_0}u_0+\int_0^t e^{(t-s)A_0} 
 \left(\alpha u(s)+f(s,u(s))  \right)ds\,,\;\;\;\;\;\forall u\in {\mathcal{H}}(Q_{T}). 
\end{equation*}
\textit{STEP. 1 \;\;We prove that 
the map $\Phi$ is well defined for some $T$.}\\
Fix $u\in {\mathcal{H}}_M(Q_{T}).$
Let us consider
$y(t):=\Phi\left(u\right)(t),$ 
then, keeping in mind Proposition \ref{Ball_mild}, we can see the function $y$ as the solution of the 
problem
\begin{equation}\label{NLmild0}
 \left\{\begin{array}{l}
\displaystyle{y^\prime(t)=A_0\,y(t)+ \left(\alpha
u(t)+f(t,u(t))  \right)
\,,\qquad  t\in[0,T] }\\ [2.5ex]
\displaystyle{y(0)=u_0 \in H^1_a(-1,1) 
}~.
\end{array}\right.
\end{equation}
By Proposition \ref{f in L2} we deduce that $f(\cdot,\cdot,u)\in L^2(Q_{T}),$ thus 
$$g(t):=\psi(t,u(t))\in L^2(0,T;L^2(-1,1)).$$
Then, applying Proposition \ref{MaxReg} we deduce that there exists a unique solution $y\in {\mathcal{H}}(Q_{T})$ of (\ref{NLmild0})  such that
$$\|y\|_{{\mathcal{H}}(Q_{T})}\leq C_0(T) \left(\|u_0\|_{1,a}+\|\psi(\cdot,u(\cdot))\|_{L^2(Q_{T})}\right).$$
Thus, keeping in mind that by Proposition \ref{MaxReg} we have $C_0(T)\leq C_0(1)$ since $T\leq1,$ applying Lemma \ref{sob3} and Proposition \ref{f in L2} there exists $T_0(M)\in(0,1)$ such that
\begin{multline*}
\!\!\!\! \|\Phi(u)\|_{{\mathcal{H}}(Q_{T})}=\|y\|_{{\mathcal{H}}(Q_{T})}\leq
 C_0(1) \left(\|\alpha\|_\infty\|u\|_{L^2(Q_T)}+\|f(\cdot,\cdot,u)\|_{L^2(Q_{T})}+\|u_0\|_{1,a}\right)\\
 \leq 2C_0(1)M, \; \forall T\in[0,T_0(M)].
\end{multline*}
Then,
$\Phi u\in{\mathcal{H}}_M(Q_{T}),\; \forall T\in[0,T_0(M)].$ 
\\
\textit{STEP. 2\; We prove that exists
$T_M
$ such that the map $\Phi$ is a contraction.}\\
Let $T\in(0,T_0(M)].$ 
  Fix $u,v\in {\mathcal{H}}_M(Q_T)$ and
set
$$w:=\Phi(u)-\Phi(v)=\int_0^t e^{(t-s)A_0} 
 \left[\alpha\left(u(s)-v(s)\right)+\left(f(s,u(s))-f(s,v(s))\right)\right]ds$$ 
then, keeping in mind Proposition \ref{Ball_mild}, 
 $w$ is solution of the following problem
\begin{equation}\label{LA4W}
\begin{cases}
w_t-(a\,w_x)_x=\psi(t,u)-\psi(t,v)  \;\;\; \mbox{ in }\; Q_T
\\
B.C.
\\
w(x,0)=0\;.\qquad\qquad\qquad  
\end{cases}
\end{equation}
By Lemma \ref{f in L2} $f(\cdot,u)\in L^2(Q_T)
$ and applying Proposition \ref{MaxReg} we deduce that
a unique solution $w\in {\mathcal{H}}(Q_T)$ of (\ref{LA4W}) exists and we have
\begin{equation}\label{Wcontr}
\|\Phi(u)-\Phi(v)\|_{{\mathcal{H}}(Q_T)}=\|w\|_{{\mathcal{H}}(Q_T)}\leq
 C_0(1) 
 \|\psi(\cdot,u)-\psi(\cdot,v)\|_{L^2(Q_T)}.
\end{equation}
From \eqref{Wcontr} using \eqref{lip_psi} and applying Lemma \ref{sob3} and Proposition \ref{f in L2} (for similar estimates see also Lemma B.1 in \cite{F1}), there exists $T_M\in(0,T_0(M))$ such that
$\Phi$ is a contraction map.
Therefore, $\Phi$ has a unique fix point in ${\mathcal{H}}_M(Q_{T_M}),$ from which the conclusion follows.\\
\end{proof}

\begin{remark}\label{s_global}
Using a classical \lq\lq semigroup'' result (see, e.g., \cite{Pazy}), applying Lemma \ref{AppL1} and the \lq\lq a priori estimate'' contained in Lemma \ref{esist glob0}, we directly obtain the global existence of the strict solution to $(\ref{PS}).$ That is, following the proof of Lemma \ref{AppL1} the unique mild solution, given by Proposition \ref{mild_LY}, is also a strict solution.  Thus, we obtain the proof of Theorem \ref{exB} in the case of a {\it static} reaction coefficient $\alpha\in L^\infty(-1,1)$.
\end{remark}

\subsubsection{Regularity of the mild solution to \eqref{PS} with initial data in $H^1_a(-1,1)$}\label{strictth0}

Now, we can give the complete proof of Theorem \ref{exB} in the general case when $\alpha\in L^\infty(Q_T)$ is a piecewise static function.
\begin{proof}[Proof of Theorem \ref{exB}]
Let us consider the problem \eqref{PS} under the assumptions $(SL)$ and $(Deg).$ Let us assume that $u_0\in H^1_a(-1,1)$ and $\alpha\in L^\infty(Q_T)$ is a piecewise static function (or a {\it simple function} with respect to the variable {\it t}), in the sense of Definition \ref{piece}, that is, there exist $m\in\N,$ $\alpha_k(x)\in L^\infty(-1,1)$ and $t_k\in [0,T], \,t_{k-1}<t_k,\, k=1,\dots,m$ with $t_0=0 \mbox{ and } t_m=T,$ 
such that $$\alpha(x,t)=\alpha_1(x)\1_{[t_{0},t_1]}(t)+\sum_{k=2}^m \alpha_k(x){\1}_{(t_{k-1},t_k]}(t),$$ where ${\1}_{[t_{0},t_1]}\,  \mbox{  and  }  \,{\1}_{(t_{k-1},t_k]}$ are the indicator function of $[t_{0},t_1]$ and $(t_{k-1},t_k]$, respectively.
Let $u\in C([0,T];L^2(-1,1))
$ the unique mild solution of \eqref{PS} with initial state $u_0\in H^1_a(-1,1),$ given by Proposition \ref{mild_LY}
\begin{equation*}\label{contr}
   u(t):=e^{tA_0}u_0+\int_0^t e^{(t-s)A_0} 
 \left(\alpha(s)u(s)+f(s,u(s))  \right)\,ds\,,\;\forall t\in [0,T],
\end{equation*}
then, 
$u,$ for $k=1,\cdots,m$, is the solution of the following $m$
problems 
\begin{equation}\label{NLmild0}
 \left\{\begin{array}{l}
\displaystyle{U^\prime(t)=A_0\,U(t)+ \alpha_kU(t)+f(t,U(t))  
\,,\qquad  t\in [t_{k-1},t_k] }\\ [2.5ex]
\displaystyle{U(t_k)=u(t_k) 
\qquad\qquad\qquad\qquad \qquad\qquad \quad   k=1,\cdots,m}~.
\end{array}\right.
\end{equation}
Since $u_0\in H^1_a(-1,1)$ and $\alpha_k\in L^\infty(-1,1)$ $(k\in\{1,\cdots,m\})$ is static on $[t_{k-1},t_k],$ 
then applying $m$ times Remark \ref{s_global} we obtain that the unique mild solution $u$ is also strict on $[t_{k-1},t_k],$ then the \lq\lq new'' initial condition $u(t_k)$ will belong to $H^1_a(-1,1)$. Thus, by iteration from $[0,t_1]$ to $[t_{m-1},t_m]$  we can complete the proof and we obtain that the mild solution $u\in{\mathcal{H}}(Q_{T})$ and it is also a strict solution on $[0,T]$.
\end{proof}


\begin{remark}
Keeping in mind Proposition \ref{mild_LY} it follows that we can extend Theorem \ref{exB} to the general case $\alpha\in L^\infty(Q_T).$ Namely, in that case $\alpha$ is {\it strongly measurable}, in the sense of the condition $(\psi SM),$ moreover we can generalize the proof of Theorem \ref{exB} from $\alpha$ piecewise static function to $\alpha$ strongly measurable.
\end{remark}


\begin{remark}
To prove Theorem \ref{exB} one can also follow a different approach, developed in the literature, in particular, by Kato in \cite{Kato} and by Evans in \cite{Evans} (see also Pazy's book \cite{Pazy}). This approach consist to consider in \eqref{NLmild} directly as operator $A(t):=A_0+\alpha(t)I,$ in which the dependence on t is discontinuous, instead of the simple operator $A_0$ that is constant in $t.$ 
\end{remark}
\subsection{Existence and uniqueness of the strong solution to \eqref{PS}}\label{strongth}
In this section we recall the proof of Theorem \ref{strong th}, given in \cite{F1} for the $(SDeg)$ case, and in \cite{F2} for the $(WDeg).$

\begin{proof}[Proof of Theorem \ref{strong th}]
Since $u_0\in L^2(-1,1),$ there exists $\{u^0_k\}_{k\in\N}\subseteq H^1_a(-1,1)$ such that 
$\displaystyle \lim_{k\rightarrow\infty}u_k^0= u_0$ in $L^2(-1,1).$
For every $k\in\N,$ we consider the following problem
\!\!\!\!\!\!\!\!\!\!\!\!\!\!\!\!\!\!\!\!\!\!\!\!\!\!\!\!\!\!
\begin{equation}\label{Pk}
\left\{\begin{array}{l}
\displaystyle{u_{kt}-(a(x) u_{kx})_x =\alpha(x,t)u_k+ 
f(x,t,u_k)\mbox{ \:\;\;\;\;\qquad\qquad a.e. \,in}\,Q_T}\\ [2.5ex]
\displaystyle{B.C.
}\\ [2.5ex]
\displaystyle{u_k(x,0)=u^0_k (x) \,\qquad\qquad\quad\qquad\qquad\qquad\qquad\qquad\;\; \quad\,x\in(-1,1)}~.
\end{array}\right.
\end{equation}
For every $k\in\N,$ by Theorem \ref{exB}
 there exists a unique $u_k\in{\mathcal{H}(Q_T)}$ strict solution to \eqref{Pk}.
Then, we consider the sequence $\{u_k\}_{k\in\N}\subseteq{\mathcal{H}(Q_T)}$ and by direct application of the Proposition \ref{uni} 
 we prove that $\{u_k\}_{k\in\N}$ is a \textit{Cauchy} sequence in the Banach space $\mathcal{B}(Q_T)$. Then, there exists $u\in{\mathcal{B}(Q_T)}$ such that, as $k\rightarrow\infty,$ $u_k\longrightarrow u$ in ${\mathcal{B}(Q_T)}$
and $\displaystyle u(\cdot,0)\stackrel{L^2}{=}\lim_{k\rightarrow\infty}u_k(\cdot,0)\stackrel{L^2}{=}u_0.$
So, $u\in {\mathcal{B}(Q_T)}$ is a strong solution.\\
 The uniqueness of the strong solution to \eqref{PS} is a direct consequence of Proposition \ref{uni}.
 \end{proof}

\section*{Acknowledgment}
The author is indebted to Prof. Ildefonso Diaz both for the suggestions about mathematical point of view for \lq\lq \textit{Energy Balance Climate Models}'', and for the cri\-ticism and the useful comments which have made this paper easier to read and understand.
The author thanks also Prof. Enrique Zuazua and the DyCon ERC team at the University of Deusto in Bilbao (in particular Umberto Biccari and Dario Pighin) for the useful and interesting discussions about the \textit{control issues} related to this paper, when the author was visitor at DeustoTech in Bilbao, partially supported by the ERC DyCon.
This work is also supported by the Istituto Nazionale di Alta Matematica (IN$\delta$AM),
through the GNAMPA Research Project 2019. 
Moreover, this research was performed in the framework of the 
French-German-Italian Laboratoire International Associ\'e (LIA), named COPDESC, on Applied Analysis,  issued by CNRS, MPI and IN$\delta$AM.

\end{document}